\documentclass[preprint,12pt]{elsarticle}

\usepackage{amssymb}
\usepackage{amsmath}
\usepackage{amsthm}
\usepackage{enumitem}
\usepackage{hyperref}
\usepackage{graphicx} 
\usepackage{amsmath}
\usepackage{amsthm}
\usepackage{xcolor}
\usepackage{multicol}
\usepackage{wrapfig}
\usepackage{algpseudocode}
\usepackage{algorithm}
\usepackage{parskip}
\usepackage{amsfonts}
\usepackage{csquotes}
\usepackage{hyperref}

\usepackage{enumitem}
\usepackage{marvosym}
\usepackage{blkarray}
\usepackage{verbatim}
\usepackage{svg}
\usepackage[normalem]{ulem}

\theoremstyle{definition}
\newtheorem{definition}{Definition}
\newtheorem{lemma}{Lemma}

\newtheorem{theorem}{Theorem}

\theoremstyle{remark}
\newtheorem*{remark}{Remark}
\newtheorem{example}[theorem]{Example}

\theoremstyle{plain}

\newenvironment{decisionproblem}[1]%
	{\vskip\topsep\noindent\textsc{#1.\/}}%
	{\par\vskip\topsep}%
 
\newcommand{\ZZ}{\ensuremath{\mathbb{Z}}}

\DeclareMathOperator{\im}{im}

\journal{Journal of Molecular Graphics and Modelling}

\begin{document}

\begin{frontmatter}



\title{Algorithmic Pot Generation: Algorithms for the Flexible-Tile Model of DNA Self-Assembly}

\author[ICERM,Rose]{Jacob Ashworth}
\author[ICERM,Tufts]{Luca Grossmann}
\author[ICERM,Johns Hopkins University]{Fausto Navarro\corref{cor1}}
\author[ICERM,Stonehill]{Leyda Almodovar}
\author[ICERM,Lewis]{Amanda Harsy}
\author[ICERM,CSU]{Cory Johnson}
\author[ICERM,Converse]{Jessica Sorrells}

\cortext[cor1]{Corresponding author. Email: fnavarro0806@gmail.com}

\affiliation[ICERM]{organization={Institute for Computational and Experimental Research in Mathematics},
            addressline={121 S Main St}, 
            city={Providence},
            postcode={02903}, 
            state={RI},
            country={USA}}

\affiliation[Rose]{organization={Department of Mathematics, Rose-Hulman Institute of Technology},
             addressline={5500 Wabash Ave},
             city={Terre Haute},
             postcode={47803},
             state={IN},
             country={USA}}

\affiliation[Tufts]{organization={Department of Mathematics, Tufts University},
             addressline={419 Boston Ave},
             city={Medford},
             postcode={02155},
             state={MA},
             country={USA}}

\affiliation[Hopkins]{organization={Department of Mathematics, Johns Hopkins University},
             addressline={555 Pennsylvania Ave},
             city={Baltimore},
             postcode={21218},
             state={MD},
             country={USA}}

\affiliation[Converse]{organization={Department of Mathematics, Converse University},
             addressline={580 E. Main Street},
             city={Spartanburg},
             postcode={29302},
             state={SC},
             country={USA}}

\affiliation[Stonehill]{organization={Department of Mathematics, Stonehill College},
             addressline={320 Washington St},
             city={North Easton},
             postcode={02357},
             state={MA},
             country={USA}}

\affiliation[Lewis]{organization={Department of Mathematics, Lewis University},
             addressline={1 University Pkwy},
             city={Romeoville},
             postcode={60446},
             state={IL},
             country={USA}}

\affiliation[CSU]{organization={Department of Mathematics, California State University - San Bernardino},
             addressline={5500 University Pkwy},
             city={San Bernardino},
             postcode={92407},
             state={CA},
             country={USA}}

\begin{abstract}

Recent advancements in microbiology have motivated the study of the production of nanostructures with applications such as biomedical computing and molecular robotics. One way to construct these structures is to construct branched DNA molecules that bond to each other at complementary cohesive ends. One practical question is: given a target nanostructure, what is the optimal set of DNA molecules that assemble such a structure? We use a flexible-tile graph theoretic model to develop several algorithmic approaches, including a integer programming approach. These approaches take a target undirected graph as an input and output an optimal collection of component building blocks to construct the desired structure.

\end{abstract}

\begin{graphicalabstract}
\includegraphics[width=\linewidth]{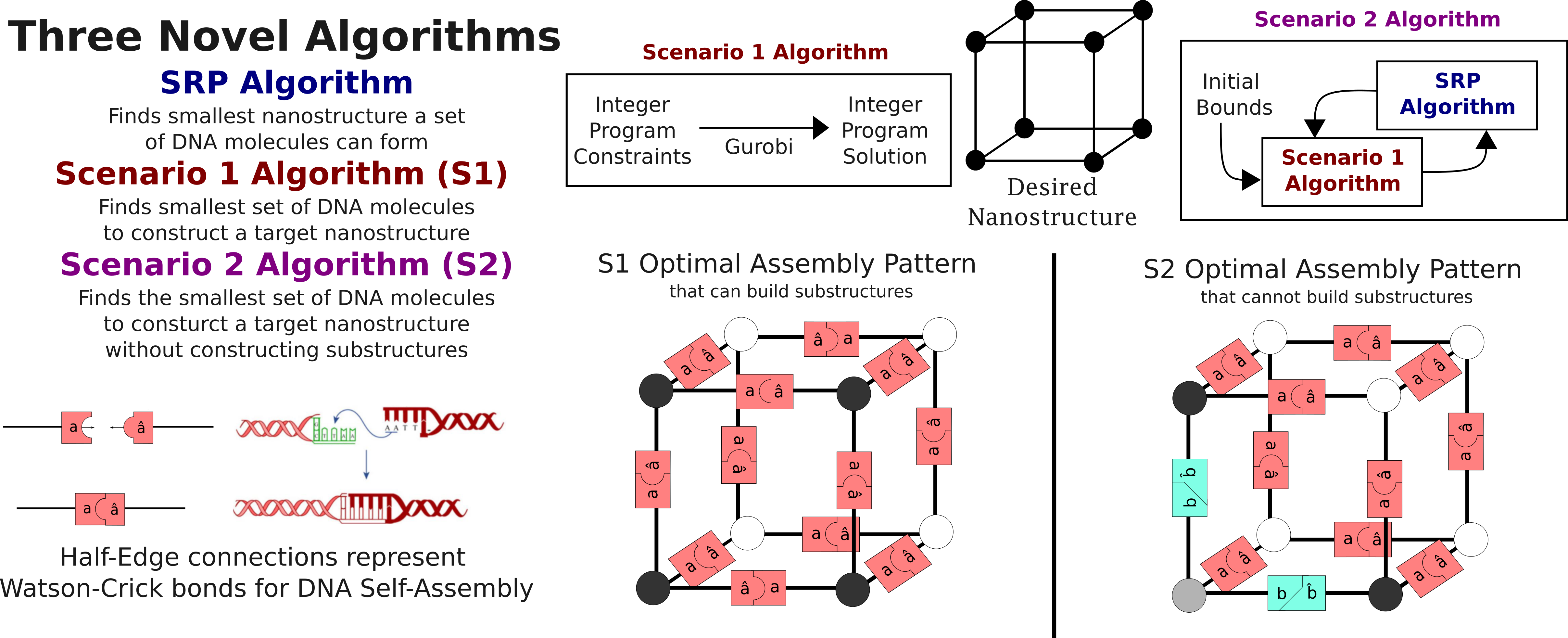}
\end{graphicalabstract}

\begin{highlights}
\item Integer programming solves the Subgraph Realization Problem
\item Integer programming finds optimal assembly designs for a given structure
\item Many new optimal assembly designs are computationally generated and verified
\end{highlights}

\end{frontmatter}

\section{Introduction}
\label{section:introduction}
Recent advancements in microbiology have motivated the study of the production of nanostructures with applications from drug delivery to molecular robotics and biomedical computing \cite{ adleman1994molecular, labean2007constructing, Lund_Manzo_2010a,Omabegho_Sha_Seeman_2009, seeman2007overview, yan2003dna}.
As these structures become smaller and more complex, production becomes increasingly difficult. As a result, laboratories exploit the Watson-Crick complementary base-pairing properties of DNA strands to create synthetic DNA sub-components with the ability to self-assemble into a target structure \cite{Jonoska_Seeman_2012, Wang}. Determining which molecules will construct a target DNA nanostructure is a critical step in the production of many microbiological tools. 

One prominent mathematical model of DNA self-assembly is referred to as the flexible-tile model, which was first introduced in \cite{jonoska2006spectrum}. This model refers to geometrically unrestrained abstractions of branched-junction DNA molecules as ``tiles." Branched junction molecules are star-shaped, consisting of several arms with `cohesive-ends' that can only bond with complementary cohesive-ends, representing complementary sequences of base pairs (see Figure \ref{fig:tile}). In \cite{ellis2014minimal}, the flexible-tile model is further developed as a graph theoretical problem in which each vertex in a graph represents a branched-junction molecule and each edge is a pair of complementary cohesive-ends that have bonded. The flexible-tile model creates a framework in which to explore two design questions pertinent to laboratories: (1) Given a desired target nanostructure, how many distinct molecule types are required for self-assembly of the target? (2) Given a collection of molecule types, what types of nanostructures can be self-assembled from them?  In \cite{ellis2014minimal} Ellis-Monaghan et. al. also introduce three ``scenarios," representing increasingly constrained levels of laboratory requirements, in which to consider these questions. These scenarios categorize the extent to which creation of byproducts (i.e. structures different from the target structure) are tolerated in the self-assembly process. We recap the relevant mathematical details of the model in 
Section \ref{subsec:model}.

In this paper, we develop algorithms to find the optimal multi-set of branched-junction molecules realizing a target DNA nanostructure. This requires engagement with two main combinatorial problems, both of which were shown to be NP-hard in \cite{sorrels}. 

The first problem is to determine whether a chosen set of branched-junction molecules could self-assemble into a nanostructure of a smaller order than the target structure. This is known as the Subgraph Realization Problem (SRP). A previous algorithm introduced in \cite{sorrels} addresses the SRP in a limited range of cases. Motivated by that work, we formulate the SRP as an integer program In Section \ref{subsec:srp}, allowing us to solve it for arbitrary numbers of molecules under feasible computational constraints. Our integer programming based approach is capable of finding minimal tile sets for arbitrary pots with dozens of tiles and bond-edge types in less than a tenth of a second (see Table \ref{table:SRPVerify}).

The second problem, which we refer to as the Optimal Pot Problem (OPP), is to determine whether a given set of molecules assembles into a graph. We construct another integer program that achieves this goal in Section \ref{subsec:lp_opp}. We use this integer program in Section \ref{subsec:s1_alg} to solve the problem in the ``scenario" where all byproducts are tolerated. We then combine this integer program with the SRP integer program to address a more difficult version of the problem in which no byproduct structures of smaller size than the target structure are permitted to form from the proposed set of molecules in Section \ref{subsec:s2_alg}. As a result, we are able to computationally determine minimal sets of tiles, i.e. abstract representations of molecules, for a wide variety of graphs in two of the three ``scenarios" of restriction described in \cite{ellis2014minimal}.

The algorithms presented in this paper (shown in Figure \ref{fig:all_algo}) can serve not only as a tool for laboratories seeking to produce target graph structures, but also for researchers looking to generalize patterns in families of graphs. Section \ref{sec:results} discusses the implementation details, including verification of known optimal solutions and reporting of new solutions found by the algorithm for a variety of graphs.

\begin{figure}
    \centering
    \includegraphics[width=6in]{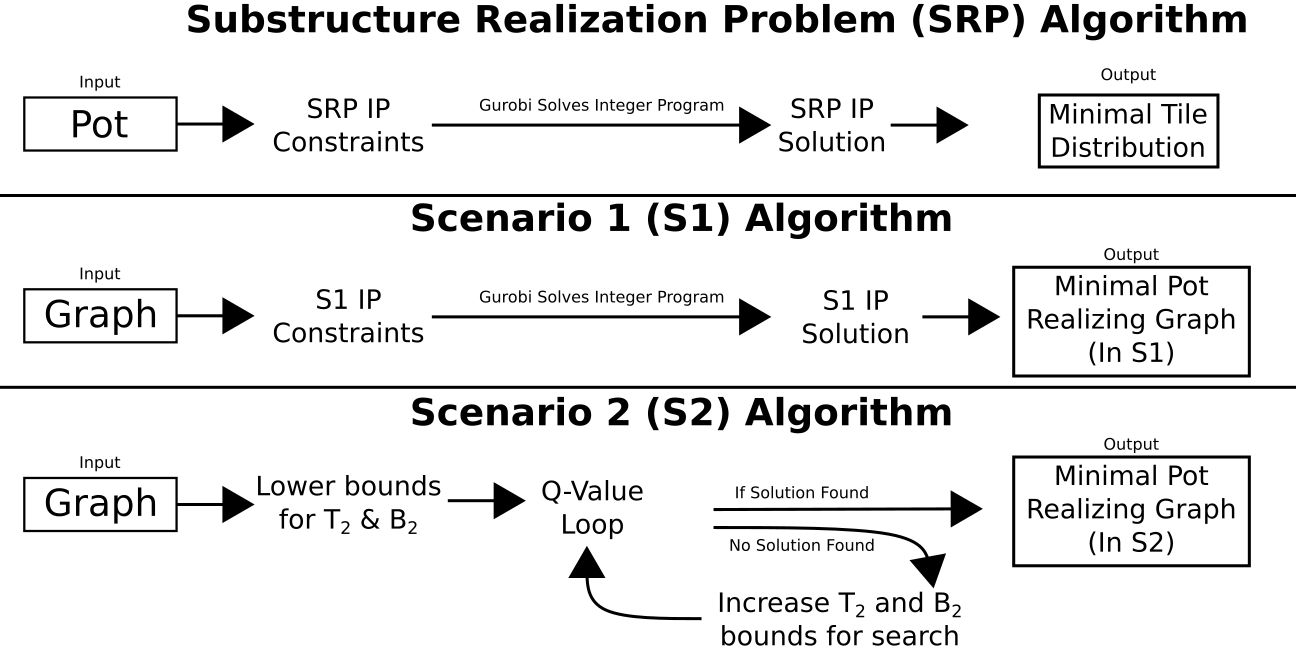}
    \caption{Overview of the three algorithms presented in this paper. The SRP algorithm is addressed in Section \ref{subsec:srp}, the S1 algorithm is addressed in Section \ref{subsec:s1_alg}, and the S2 algorithm in Section \ref{subsec:s2_alg}.}
    \label{fig:all_algo}
\end{figure}

\section{Methods}
\label{sec:methods}

\subsection{Graph Theoretical Formalism}
\label{subsec:gtf}
Throughout the paper, we will denote the set of nonnegative integers as $\ZZ^+$ and the set of nonnegative integers less than or equal to $k$ as $\ZZ_k$.

A discrete graph $G$ consists of a set $V=V(G)$ of vertices and a set $E=E(G)$ of edges together with a map $\mu:E \rightarrow V^{(2)}$  where $V^{(2)}$ is the set of (not necessarily distinct) unordered pairs of elements of $V$. If $\mu(e)=\{u,v\}$, then $u$ and $v$ are the vertices incident with edge $e$. We allow graphs to have loops and multiple edges, so that it is possible for $\mu(e) = \{u,v\}$ with $u=v$ (in the case of a loop edge) or for $\mu(e)=\mu(f)$ (in the case of multiple edges). We use the notation $\#$ to denote size or quantity, so that e.g. $\#V(G)$ will denote the order of $G$.

We denote a \emph{half-edge} of a vertex $v$ as $(v,e)$ if $v \in \mu(e)$. $H$ shall denote the set of half edges of $G$. 

For vertex $v_i$, we denote its degree (the number of incident edges to $v_i$) as $d_i$. For a graph $G$, we denote the maximum vertex degree as $\Delta(G)$ and the minimum vertex degree in $G$ as $\delta(G)$. We denote the neighborhood of vertex $v$, the set of all vertices adjacent to $v$, as $N_G(v)$. 

When two graphs $G_1$ and $G_2$ are isomorphic, we denote this as $G_1 \cong G_2$.

\subsection{Graph Assembly Model} \label{subsec:model}

To model the process of DNA self-assembly, we employ the graph-theoretical formalism described in  \cite{ellis2014minimal, ellis2019tile}. For the convenience of the reader, we include some of the relevant definitions in this section and introduce several new convenient tools for our purposes.

The $k$-armed branched junction molecule is the basic unit of this model. A \emph{k-armed branched junction molecule} is a star-shaped molecule whose arms are formed from strands of DNA (see left image in Figure \ref{fig:tile}). At the end of each of these arms is a region of unpaired bases, forming a \emph{cohesive-end} (also known informally as a \emph{sticky end}). We encode cohesive ends using symbols.

\begin{figure}[!ht]
\centering
\includegraphics[width=.8\linewidth]{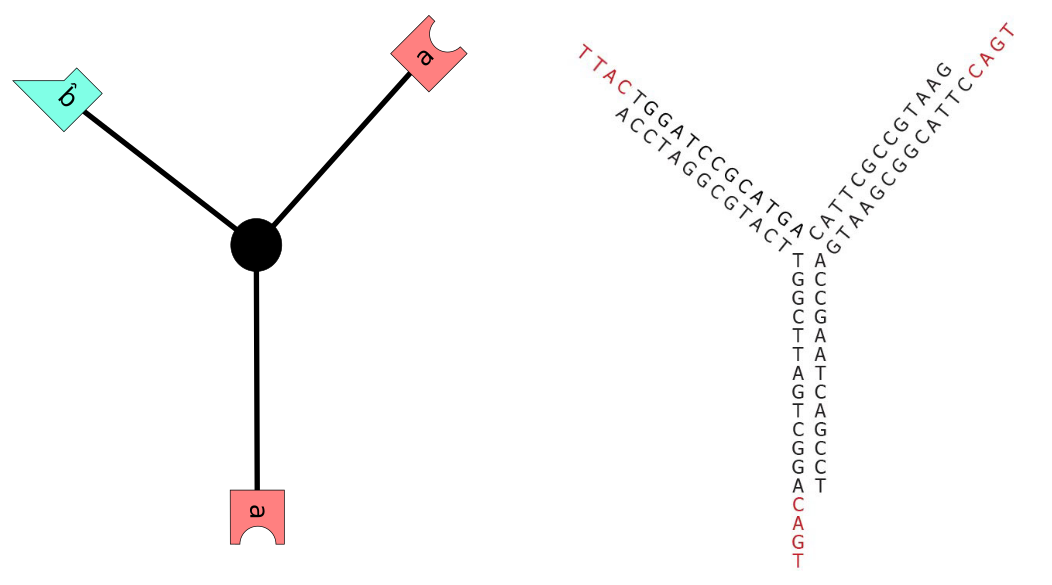}
\caption{The tile $t = \{a^2,\hat{b}\}$ (set notation) or $aaB$ (string notation), and its corresponding $k$-armed branched junction molecule}
\label{fig:tile}
\end{figure}

\begin{definition} \label{def:alphabet} A \emph{cohesive-end type} is an element of a finite set $S=\Sigma \cup \hat{\Sigma} $, where $\Sigma$ is the set of hatted symbols and $\hat{\Sigma}$ is the set of un-hatted symbols.  Each cohesive-end type corresponds to a distinct arrangement of bases forming a cohesive-end on the end of a branched junction molecule arm, such that a hatted and an un-hatted symbol, say $a$ and $\hat{a}$, correspond to complementary cohesive-ends.  We do not allow palindromic cohesive-ends,  so a cohesive-end corresponding to $a$ is complementary to but different from the cohesive-end corresponding to $\hat{a}$ for all $a\in \Sigma$. Moreover we use the convention that $\hat{\hat a} =a$.  \end{definition}

\begin{definition} A cohesive-end type joined to its complement forms a \emph{bond-edge type}, which we identify by the un-hatted symbol. That is, $\Sigma$ represents the set of all bond-edge types. \end{definition}

\begin{example} Cohesive-end types $a$ and $\hat{a}$ may join to form a bond-edge of type $a$. This process is shown in Figure \ref{fig:c_4realization}. \end{example}

Arms with complementary cohesive-ends can bond via Watson-Crick base pairing. In the simplest setting, the arms of a branched junction molecule are double stranded DNA with one strand extending beyond the other to form a cohesive region (see \cite{SK94}). In this case, the molecules are quite flexible. The mathematical model we employ here assumes there are no geometric restrictions, for example on the inter-arm angles or arm lengths. See \cite{ferrari2018} for a model that encompasses such geometric constraints.

In this graph theoretical framework, we encode a $k$-armed branched junction molecule and its un-bonded cohesive ends as follows.

\begin{definition} A $k$-armed branched junction molecule is represented by a vertex of degree $k$ with $k$ incident half-edges called a \emph{tile}. Thus,  a tile type \textit{t} is a classification of a flexible-armed branched junction molecule according to its set of cohesive-end types. The half-edges of a tile are labeled by the cohesive-end types corresponding to the cohesive-ends on the arms of the molecule the tile represents. The tile type of a tile is thus given by the multi-set of its cohesive-end types whose multiple entries of the same cohesive-end type are indicated by the exponent to the corresponding symbol. \end{definition} 

For simplicity of computation, throughout this paper we often write a tile as a lexicographically ordered string.

\begin{example} A tile of degree $\mathrm{3}$ with cohesive ends denoted as $a, a, $ and $\hat{b}$ is pictured in Figure \ref{fig:tile}. This tile would typically be written as $\{a^2,\hat{b}\}$, but here we use $aaB$, where $B$ is used in place of $\hat{b}$.
\end{example}

Tiles as we have defined them here are sometimes referred to as `flexible tiles' to avoid conflation with rigid tiles seen in Wang tilings or tile assembly models (TAM) (see \cite{CW2017}). Note that a self-assembled structure may contain several tiles of the same tile type. For greater mathematical simplicity, we assume that once a lab creates a particular tile, an unlimited number of tiles of that type is available to be used in the self-assembly process. To these ends, we define a working collection of tiles as the set of distinct tile types it contains.

\begin{definition} A \emph{pot} is a collection of tiles. More precisely, a \textit{pot} is given by its collection of tile types, and this is represented by a set $P = \{t_1,...,t_k\}$ where each $t_i$ is a tile type for some tile in $P$, and for all $a\in \Sigma$, if there is $i$ such that $a \in t_i$ then there is $j \in\{1,...,k\}$ such that $\hat{a} \in t_j$. We arbitrarily choose an ordering of the tiles, although all orderings represent the same pot. The set of bond-edge types that appear in tile types of $P$ is denoted with $\Sigma(P)$, and we write $\#P$ to denote the number of distinct tile types in $P$ and $\#\Sigma(P)$ to denote the number of distinct bond-edge types that appear in $P$. \end{definition}

\begin{definition} An \textit{assembly design} of a graph $G$ is a labeling $\lambda : G \rightarrow \Sigma \cup \hat{\Sigma}$ of the half-edges of $G$ with the elements of $\Sigma$ and $\hat{\Sigma}$ such that if $e \in E(G)$ and $\mu(e) = \{u,v\}$, then $\widehat{\lambda((v,e))} = \lambda((u,e))$. This means that each edge of $G$ receives both a hatted and an un-hatted version  of the same symbol, one on each of its half-edges. \label{def:assembly_design} \end{definition}

\begin{definition} Given an assembly design $\lambda$, the \textit{labelled degree} $d_{\sigma_i}(v)$ denotes the number of half edges incident to vertex $v$ which are labelled with cohesive-end type $\sigma_i$ (i.e. such that $\lambda((v,e)) = \sigma_i \in \Sigma \cup \hat{\Sigma})$ \label{def:ldeg}.
\end{definition}

\begin{definition} An \emph{assembling pot $P_{\lambda}(G)$} for a graph $G$ with assembly design $\lambda$ is the set $P_{\lambda}(G) = \{t_v \mid v \in V(G)\}$ where $t_v =\{ \lambda (v,e) \mid v \in \mu(e), e \in E(G)\}$.\end{definition}

\begin{definition} We say a pot $P$ \textit{realizes} the graph $G$ if there exists an assembly design $\lambda : H \rightarrow \Sigma \cup \hat{\Sigma}$ such that $P_{\lambda}(G) \subseteq P$. \label{def:realize} \end{definition} 

\begin{definition} The set of graphs with their associated assembly designs realized by a pot $P$, namely $\{(G,\lambda): P_{\lambda}(G) \subseteq P \}$, is called the \textit{output} of $P$ and is denoted by $\mathcal{O}(P)$. \end{definition}

\begin{definition}
    Given an an assembly design $\lambda$, its \textit{realization function} $f: V(G) \to P_\lambda(G)$ maps $v$ to the tile type used to label $v$, i.e.
    \begin{equation*}
        f(v) = \{\sigma_i^{d_{\sigma_i}(v)} : i \in \ZZ_N\} \cup \{ \hat{\sigma}_i^{d_{\sigma_i}(v)} : i \in \ZZ_N\},
        \label{eq:potfromorientation}
    \end{equation*}
    where $d_{\sigma_i}(v)$ denotes the number of half edges incident to $v$ which are labelled $\sigma_i$ (see Definition \ref{def:ldeg}). \label{def:relf}
\end{definition}

\begin{example}  Figure \ref{fig:K3digraph} shows an assembly design for $K_3$. Its realization function can be defined as $f(v_1) = \{a,b\}, f(v_2)=\{a, \hat{b}\}, f(v_3)=\{\hat{a}^2\}$. 

    \begin{figure}
        \centering
        \includegraphics[width=0.4\linewidth]{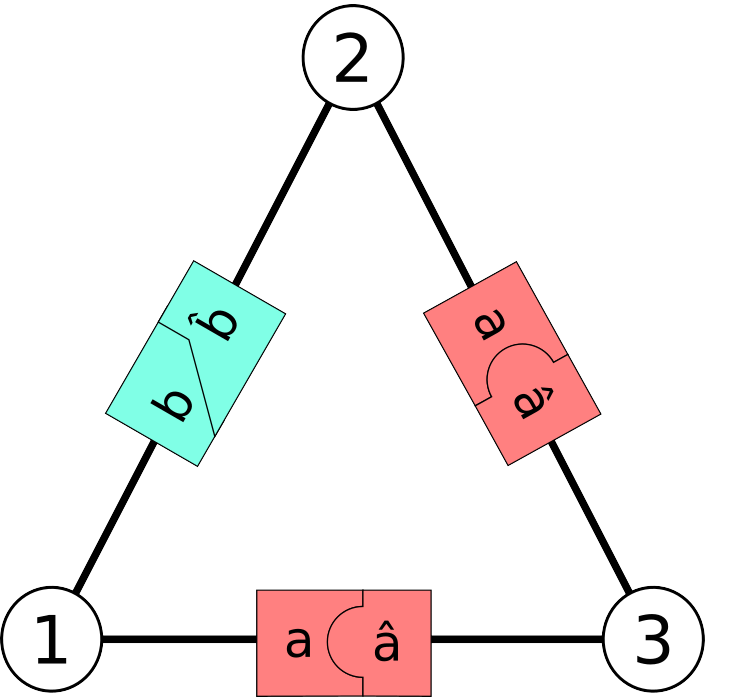}
        \caption{A visual representation of the assembly of $K_3$ with the pot $\{ab,\hat{b}a,\hat{a}^2\}$}
        \label{fig:K3digraph}
    \end{figure}
\end{example}

Within the flexible tile model of DNA self-assembly, we can consider the strategic design of pots that can assemble given graphs. An overarching goal is to find pots with the fewest numbers of tile and/or bond-edge types, and thus to theoretically determine the minimum number of branched junction molecules needed for self-assembly of certain DNA nanostructures. Given a target graph $G$, we seek minimal pots $P$ under consideration of three levels of restriction, which are the ``scenarios'' described in \cite{ellis2014minimal}. 
\begin{itemize}
    \item \emph{Scenario 1.} Least Restrictive: $G \in \mathcal{O}(P)$. Note: This allows the possibility that there exists $H \in \mathcal{O}(P)$ such that $\#V(H) < \#V(G)$. 
    \item \emph{Scenario 2.} Moderately Restrictive: $G \in \mathcal{O}(P)$ and for all $H \in \mathcal{O}(P)$, $\#V(H)\geq \#V(G)$. Note: This allows the possibility that there exists $H \in \mathcal{O}(P)$ such that $\#V(H)=\#V(G)$ but $H \not \cong G$. 
    \item \emph{Scenario 3.} Highly Restrictive: $G \in \mathcal{O}(P)$, and for all $H \in \mathcal{O}(P)$, $\#V(H)\geq\#V(G)$, and if $\#V(H)=\#V(G)$ then $H \cong G$.
\end{itemize} 

\begin{example}
    Consider the different ways of assembling the graph $K_4$ shown in Figure \ref{fig:k4_scenarios}. The left column shows a specific pot and how it realizes (self-assembles) $K_4$. This pot uses only one bond-edge type and is satisfactory in Scenario 1. At the bottom of the column is another, smaller order graph also realized by the same pot, which disqualifies the pot from Scenario 2. The middle column shows how another pot with three bond-edge types realizes $K_4$. The pot does not assemble into any smaller order graphs, so it is satisfactory for Scenario 2. At the bottom of the middle column is a graph non-isomorphic to $K_4$ with equal order to $K_4$, which disqualifies this pot from Scenario 3. Finally, the right column shows a realization of $K_4$ using a third pot. This pot cannot realize any graphs of smaller order or of equal order but non-isomorphic to $K_4$, so this pot is valid in Scenario 3.
    \begin{figure}
        \centering
        \includegraphics[width=1\linewidth]{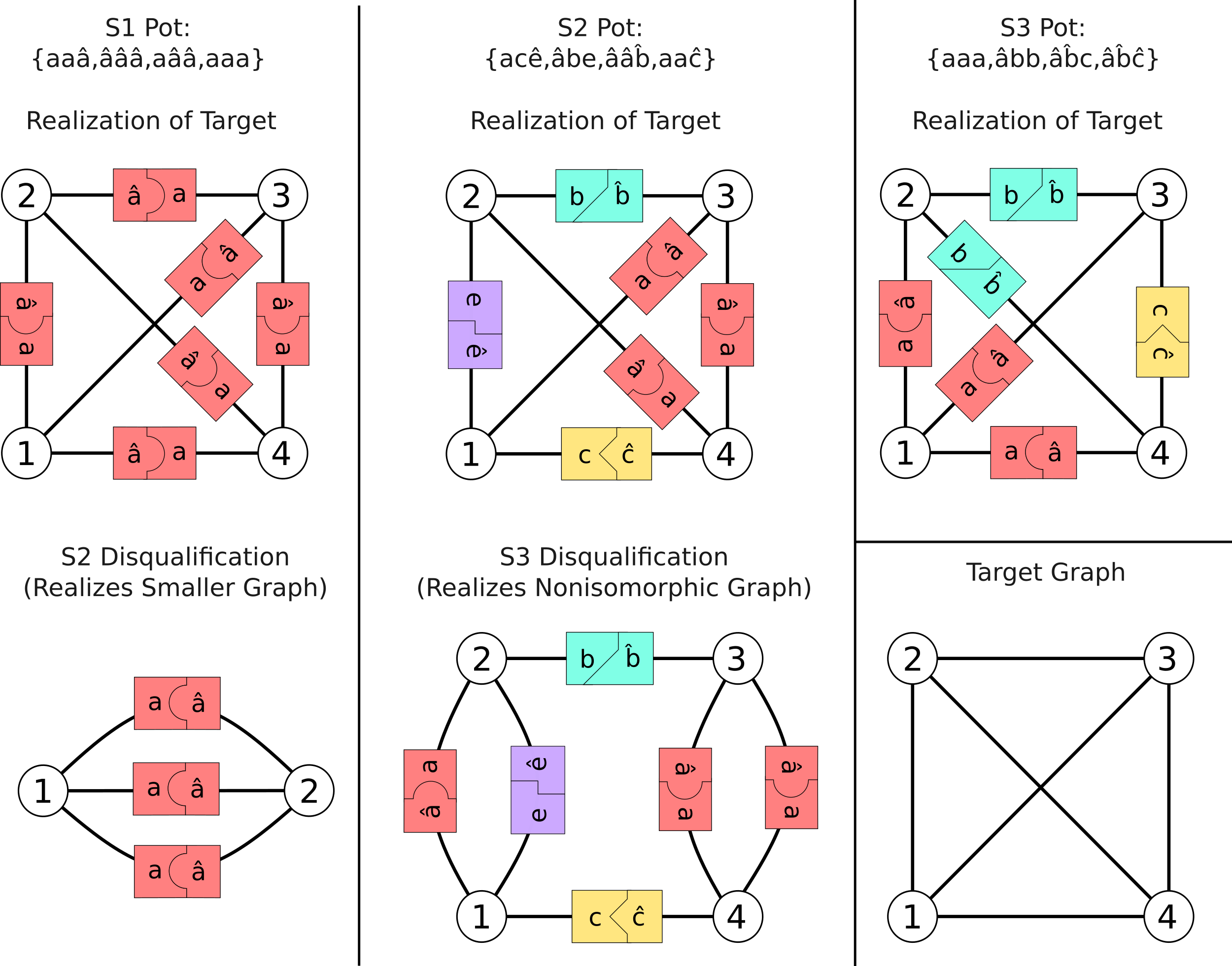}
        \caption{Illustration of three pots, each realizing $K_4$ in a different scenario, and examples of other graphs realized by those pots within the confines of the given scenario.}
        \label{fig:k4_scenarios}
    \end{figure}
\end{example}

Given the three theoretical laboratory scenarios, we can search for pots in each scenario that consist of the fewest tile types or fewest bond-edge types. This is analogous to seeking the smallest collections of molecules that will self-assemble into a given target nanostructure while also potentially avoiding byproduct structures, or waste. The following definitions provide notation for minimality with respect to tile types and bond-edge types.

\begin{definition}
$T_i(G) = \text{min}\{ \#P \mid P \text{ realizes } G \text{ according to Scenario } i\}$.
\end{definition}

\begin{definition}
$B_i(G) = \text{min}\{ \# \Sigma (P) \mid P \text{ realizes } G \text{ according to Scenario } i\}$.
\end{definition}

This paper focuses on finding optimal pots in Scenarios 1 and 2. This amounts to solving the following problem: 

\begin{decisionproblem}{Optimal Pot Problem (OPP)} Given a graph $G$ and scenario $S_i$, find pots with $T_i(G)$ tile types and/or $B_i(G)$ bond-edge types. \label{prob:OPP}\end{decisionproblem}

\begin{remark}
    It was shown in \cite{sorrels} that it is always possible to find a pot simultaneously achieving $T_1(G)$ and $B_1(G)$, but that graphs exist such that no pot will simultaneously achieve $T_3(G)$ and $B_3(G)$. It was conjectured that there existed no pot realizing the cube graph and simultaneously achieving $T_2(G)$ and $B_2(G)$, but this was disproved in \cite{ferrari}. The term ``biminimal pot" was coined in \cite{ferrari} to describe pots simultaneously achieving a minimum number of tile types and bond-edge types. The question of whether a biminimal pot satisfying Scenario 2 exists for every graph remained an open question until Toby Anderson, Olivia Greinke, Iskandar Nazhar, Luis Santori,   discovered that the $Y_{3,3}$ stacked prism graph does not have a biminimal Scenario 2 pot (pending citation). Thus, for Scenario 2 we present an algorithm to search for pots that are minimal with respect to either the number of tile types and the number of bond-edge types, with the knowledge that there might exist two distinct pots achieving these values.
\end{remark}

From Scenario 2 (and Scenario 3) arises two problems that must be grappled with when attempting to design an assembling pot: (1) determining if a given pot will realize a graph of a given order, (2) determining whether a pot that realizes a given graph can also realize any smaller graph. These two problems and their computational complexity were studied in \cite{sorrels}. Below we reproduce the formulations of these problems.

\begin{decisionproblem}{Pot Realization Problem (PRP)}
 Given a pot $P$ and a positive integer $k$, does $P$ realize a graph $G$ with $k$ vertices? \label{prob:PRP}
\end{decisionproblem}

\begin{decisionproblem}{Substructure Realization Problem (SRP)}
 Given an assembling pot $P_{\lambda}(H)$ for a graph $H$ of order $n$ and an integer $\ell$ with $0<\ell<n$, does $P_{\lambda}(H)$ realize a graph $F$ with $\ell$ vertices? \label{prob:SRP}
\end{decisionproblem}

\subsection{SRP Integer Program}
\label{subsec:srp}
Recall that in order to determine whether a pot $P$ is valid for a given target graph $G$ in Scenarios 2 or 3,  it must be known whether any graphs of order smaller than $\#V(G)$ are in $\mathcal{O}(P)$. To begin, we follow \cite{ellis2014minimal} and \cite{sorrels} to describe the original framework for the SRP and the challenges presented by certain cases. In \cite{ellis2014minimal} the augmented matrix of a linear system encoding some constraints on an assembling pot $P_{\lambda}(G)$ for a given graph $G$ is introduced as a first-line computational approach for determining whether $P$ is valid in Scenario 2. 

The following definitions and propositions relating to the \textit{construction matrix} are given in \cite{ellis2014minimal}.

\begin{definition}\label{def:matrix} Let $P = \{ t_1,...,t_p \}$ be a pot and let $z_{i,j}$ denote the net number of cohesive-ends of type $a_i$ on tile $t_j$, where un-hatted cohesive-ends are counted positively and hatted cohesive-ends are counted negatively. Then the following system of equations must be satisfied by any connected graph in $\mathcal{O}(P)$:
\begin{eqnarray*} z_{1,1}r_1+z_{1,2}r_2+...+z_{1,p}r_p &=& 0 \\
& \vdots & \\
z_{m,1}r_1+z_{m,2}r_2+...+z_{m,p}r_p &=& 0 \\
r_1 + r_2 + ... + r_p &=& 1 \end{eqnarray*}

The \emph{construction matrix} of $P$, denoted $M(P)$, is the corresponding augmented matrix:

\begin{equation*}
M(P) = \begin{bmatrix} \begin{array}{*{20}{cccc|c}}
   {{z_{1,1}}} & {{z_{1,2}}} &  \ldots  & {{z_{1,p}}} & 0  \\
    \vdots  &  \vdots  & {\ddots} &  \vdots  & {} \vdots \\
   {{z_{m,1}}} & {{z_{m,2}}} &  \ldots  & {{z_{m,p}}} & 0  \\
1 & 1 &  \ldots  & 1 & 1  \\
 \end{array} \end{bmatrix}. 
\end{equation*}
\end{definition}

\begin{definition} The solution space of the construction matrix of a pot $P$ is called the {\it spectrum} of $P$, and is denoted  $\mathcal S(P)$. \end{definition}

The construction matrix is used primarily to determine whether there exists in $\mathcal{O}(P)$ a graph of smaller order than the desired target graph. For a given solution $\langle r_1,...,r_p \rangle$ to $M(P)$, the least common denominator of the $r_i$'s gives the smallest order of a graph in $\mathcal{O}(P)$. This is detailed in the following result from \cite{ellis2014minimal}.

\begin{theorem}\label{thm:spectrum}
Let $P=\{t_1, \ldots , t_p\}$ be a pot.  Then:

\begin{enumerate}
    \item If a graph $G$ of order $n$ is realized by $P$ using $R_j$ tiles of type $t_j$, then $\frac{1}{n}\langle R_1, \ldots, R_p\rangle$ is a solution of the construction matrix $M(P)$, i.e. is in $\mathcal{S}(P)$.
    \item If $\langle r_1, \ldots, r_p \rangle \in \mathcal{S}(P)$, and $n$ is a positive integer such that $nr_j \in \mathbb{Z}_{\geq 0}$ for all $j$, then there is a graph of order $n$ in $\mathcal{O}(P)$ that is realized by using $nr_j$ tiles of type $t_j$. 
    \item  The smallest order of a graph in $\mathcal{O}(P)$ is $min \{lcm \{b_j|r_j \neq 0 \text{ and } r_j = a_j/b_j\}\}$ where the minimum is taken over all $\langle r_1,\ldots,r_p\rangle \in \mathcal{S}(P)$ such that $r_j \geq 0$ and $a_j/b_j$ is in reduced form for all $j$.
\end{enumerate}
\end{theorem}


In the cases where $\mathcal{S}(P)$ consists of a unique solution, the construction matrix quickly provides a solution to the SRP. However, if the system of equations defined by the construction matrix does not have a unique solution, this approach does not give an immediate answer to the SRP \cite{sorrels}. The spectrum $\mathcal{S}(P)$ will have degrees of freedom whenever $\#P > \#\Sigma(P) + 1$, as then the dimensions of $M(P)$ will be $(\#\Sigma(P) + 1) \times \#P$.  This occurs even in examples of quite simple target graphs and small pots, as we demonstrate in the following example.

\begin{example}\label{ex:p3} The pot $P = \{\{a\},\{\hat{a}\},\{a,\hat{a}\}\}$ realizes the graph $P_3$, the path graph of order 3. Since $\#P = 3 > 2 = \#\Sigma(P) + 1$ (i.e. $M(P)$, shown below, has two more columns than rows), the system has one degree of freedom. 


    \[ M(P) = 
    \begin{bmatrix}
    1 & -1 & 0 & 0 \\
    1 & 1 & 1 & 1
    \end{bmatrix}
    \]

    It is easy to verify,

    \begin{equation*} \mathcal{S}(P) = \left\{ \frac{1}{2r} \langle r-t, r-t, 2t \rangle \middle| \; r \in \mathbb{Z}^+,  t \in \left(\mathbb{Z}\cap [0,r]\right) \right\}. \end{equation*}
    
    Here, $r$ is a scaling factor and $t$ represents the free variable. In this instance, there exist multiple $H \in \mathcal{O}(P)$ such that $\#V(H) < \#V(G)$.  A graph of order one, a self-loop with the tile $\{a,\hat{a}\}$, can be realized with the solution corresponding to $r=t$. The path graph $P_2$ can be realized (using tiles $\{a\}$ and $\{\hat{a}\}$) with the solution corresponding to $t=0$. Note that the target graph, $P_3$, is realized in accordance with $r=3, t=1$. 
\end{example}

 As can be observed in Example \ref{ex:p3}, when $\mathcal{S}(P)$ has degrees of freedom, one must check all valid values of the free variable(s) in order to determine if there exist smaller order graphs in $\mathcal{O}(P)$. In cases of larger graphs and pots, this becomes tedious and time-consuming. In addition to these difficulties, it was shown in \cite{sorrels} that an optimal pot in Scenario 2 might correspond to multiple distinct solutions (arising from distinct values of the free variable(s)) of $M(P)$.  

In \cite{sorrels}, an hmhm which checks all possible tile distributions with non-negative integer tile quantities was proposed to handle cases where $\mathcal{S}(P)$ has degrees of freedom. In theory, the algorithm can solve the SRP for systems with any number of free variables. However, in cases with more than two free variables, the algorithm presents inhibitory runtime costs. Thus, the algorithm is only implemented for pots with one or two free variables. 

Instead, we reframe the SRP as an integer program. By using this program, we can efficiently determine the smallest order graph realized by a given pot without trying to find a solution to the construction matrix, allowing us to solve the SRP no matter how many free variables the spectrum of the pot has. 

Notice that the $R_j$ in the first part of Theorem \ref{thm:spectrum} describe the number of tiles used in a specific assembly design of $G$.

\begin{definition}
    Given a pot $P$, a \textit{tile distribution} is a tuple $\Omega(P) = (R_1, \cdots R_{\#P})$, where $R_i$ represents the number of tiles of type $t_i$.
\end{definition} 

Note that if $\Omega(P)$ describes the multi-set of tiles used to build a graph $G$, then $$\displaystyle\sum_{i=1}^{\#P}R_i = \#V(G).$$

\begin{remark}
    Let $\Omega(P) = (R_1,...,R_{\#P})$ be a tile distribution of a given pot $P = \{t_1, ... t_p\}$. Then $\Omega(P)$ corresponds to some $\langle r_1,...,r_p\rangle \in \mathcal{S}(P)$ such that
    \begin{align*}
        r_i = \displaystyle\frac{R_i}{\displaystyle\sum_{k=1}^{\#P}R_k}.
    \end{align*}
\end{remark}

\begin{figure}
        \centering
\includegraphics[width=4in]{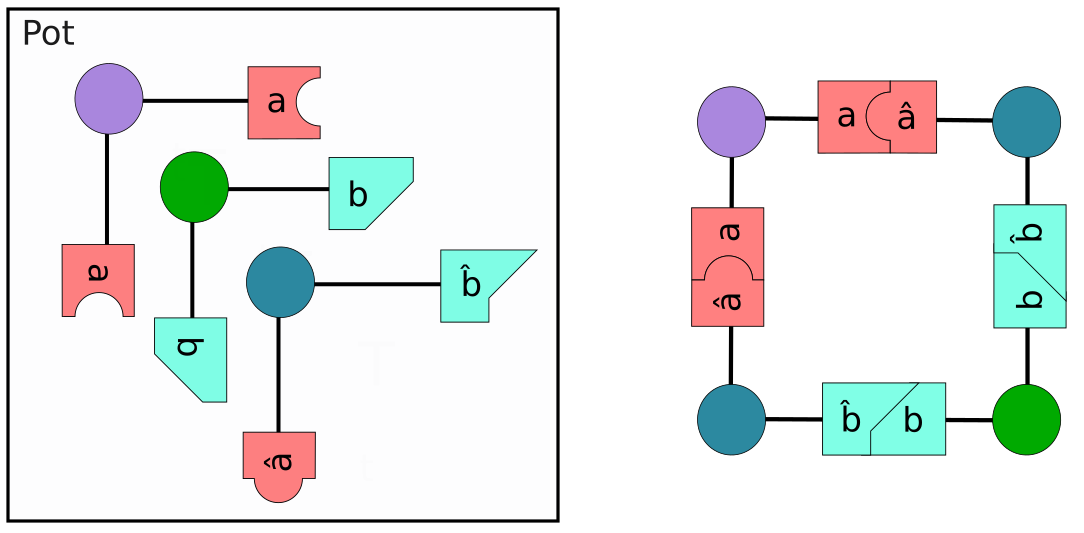}
\caption{Optimal Scenario 2 Construction of $C_4$}
\label{fig:C4RealizationS2}
\end{figure}

\begin{example}
    The tile distribution associated with the realization of $C_4$ by the pot $P = \{aa, \hat{a}\hat{b}, bb\}$ as shown in Figure \ref{fig:C4RealizationS2} is $\Omega(P) = (1, 2, 1)$.
\end{example}

Using tile distributions, the SRP can be reformulated as follows:

\begin{decisionproblem}{Substructure Realization Problem (SRP)} Given a graph $G$ of order $N$ and a pot $P$ which realizes $G$, does $P$ realize a graph of order less than $N$ with some tile distribution $\Omega(P)$? 
\end{decisionproblem}

The following lemma presents the SRP in terms of tile distributions, and thus as integer-valued equations.

\begin{lemma}\label{lemma:P_realize}
    Given a pot $P = \{t_1, \cdots, t_n\}$, $P$ realizes a graph G of order $N$ if and only if there exists a tile distribution $\Omega(P) = (R_1, \cdots, R_n)$ such that
    \begin{center}
    $\displaystyle\sum_{j=1}^nR_j = \#V(G)$ \quad and \quad$\displaystyle\sum_{j=1}^nR_jz_{i,j} = 0$ for each $i \in \{1, \dots, |\Sigma| \}.$
    \end{center} 

\end{lemma}



Finding the smallest graph that can be realized by a pot amounts to minimizing $\sum_{i=1}^nR_i$. Thus, we can see that the SRP is equivalent to the following integer programming problem.
\begin{equation}\label{lp:srp} 
\begin{aligned}  
    \min &\sum_{i=1}^nR_i \\
    \text{subject to}: &\sum_{j=1}^nR_jz_{i,j} = 0 \quad \forall i \in \ZZ_{\#P} 
\end{aligned}
\end{equation}


Section \ref{sec:results} details the computational tools we use to solve this integer program as well as example use cases. 

\subsection{Optimal Pot Problem Integer Program}
\label{subsec:lp_opp}
Now we turn to solving the Optimal Pot Problem in Scenarios 1 and 2. The primary roadblock is the PRP (\ref{prob:PRP}), i.e. being able to check whether a given pot realizes a given graph efficiently. This problem is known to be NP-hard, and since it will need to be answered for every possible pot of each size (of which there are combinatorially many), the number of necessary computations quickly becomes unreasonable. 

In this section, we reformulate the Optimal Pot Problem in Scenario 1 as an integer program. The overall approach will be to efficiently search the space of assembly designs of a given target graph $G$.

Recall that pot $P$ realizes a graph if there exists an assembly design with assembling pot containing the tiles of $P$. Since we only care about minimal pots, we can restrict our attention to assembly designs whose assembly pots are equal to $P$. In other words, a pot $P$ realizes a graph via assembly design $\lambda$ if its realization function $f$ (Definition \ref{def:relf}) is surjective; i.e. $\im f = P$.

To solve the OPP, we will express the constraint $\im f = P$ as a series of linear inequalities (see Definition \ref{def:PCC}). To do this, we must have some way of describing an assembly design $\lambda$ as a fixed length integer vector. This vector will depend on the maximum possible size of a minimal pot realizing $G$, which we denote as $\theta$, and the maximum possible number of bond-edge types in a minimal pot realizing $G$, which we denote as $\beta$. In an extension of the notation from Definition \ref{def:alphabet}, we define $\Sigma_\beta = \{\sigma_1, \sigma_2, \cdots \sigma_{\beta} \}$ and $\hat{\Sigma}_\beta = \{\hat{\sigma}_1, \hat{\sigma}_2, \cdots \hat{\sigma}_{\beta} \}$, so that $\Sigma_\beta \cup \hat{\Sigma}_\beta$ will give all the half-edge labels we are considering.  

\begin{definition}\label{def:decvar}
   Given an assembly design $\lambda$ of $G$ with realization function $f$ such that $\im f = P = \{t_1, \cdots, t_N\}$ for $N \leq \theta$, $Var(\lambda)$ is the following collection of decision variables describing $\lambda$:

\begin{itemize}
    \item $t_{n,a} \in \ZZ^+$: the number of half-edges of type $a$ on $t_n$
    \item $w_{i,a} \in \ZZ^+$: the number of half-edges of type $a$ on xx $v_i$, or $d_{\sigma_i}(v)$
    \item $x_{i,n} \in \{0, 1\}$: 1 if $f(v_i) = t_n$ ($v_i$ uses tile type $t_n$), and 0 otherwise 
    \item $k_{n} \in \{0, 1\}$: 0 if $t_n$ is an ``empty" tile. This allows us to represent pots with less than $\theta$ tiles. 
    \item $e_{i,j,a} \in \{0, 1\}$: 1 if  $\lambda((v_i, \{v_i, v_j\})) = a$ (i.e. if the half-edge $(v_i, \{v_i, v_j\})$ is labeled with cohesive-end type $a$), 0 otherwise. 
    \end{itemize}
\end{definition}

Thus, to search over the space of labeled orientations of a given graph $G$, our integer program will search over the space of tuples: \[C = \{(t_{n,a}, w_{i,a}, x_{i,n}, k_{n}, e_{i,j,a}): a\in \Sigma_{\beta} \cup \widehat{\Sigma}_{\beta}, i,j \in \ZZ_{N}, n \in \ZZ_{\theta}\}.\]
Each element of $C$ can be encoded as an integer-valued vector with non-negative entries.

\begin{example}
    Consider the realization of $L=O(C_4)$ using the pot $P = \{ t_1=\{a^2\}, t_2=\{\hat{a}^2\} \}$ pictured in Figure \ref{fig:c_4realization}. The $k_i$ and $t_{i,\sigma}$ decision variables and values are shown on the left for each tile in the pot. Select $e_{i,j,\sigma}$ decision variables are shown on the right, although most are omitted as every edge has 2 sets of variables for each $\sigma \in \Sigma\cup \hat{\Sigma}$. In the center, the $x_{i,j}$ decision variables are shown for $i=1$ and $i=3$, and arrows demonstrate how these decision variables control the mapping of tiles onto vertices.
    
    \begin{figure}
        \centering
        \includegraphics[width=1\linewidth]{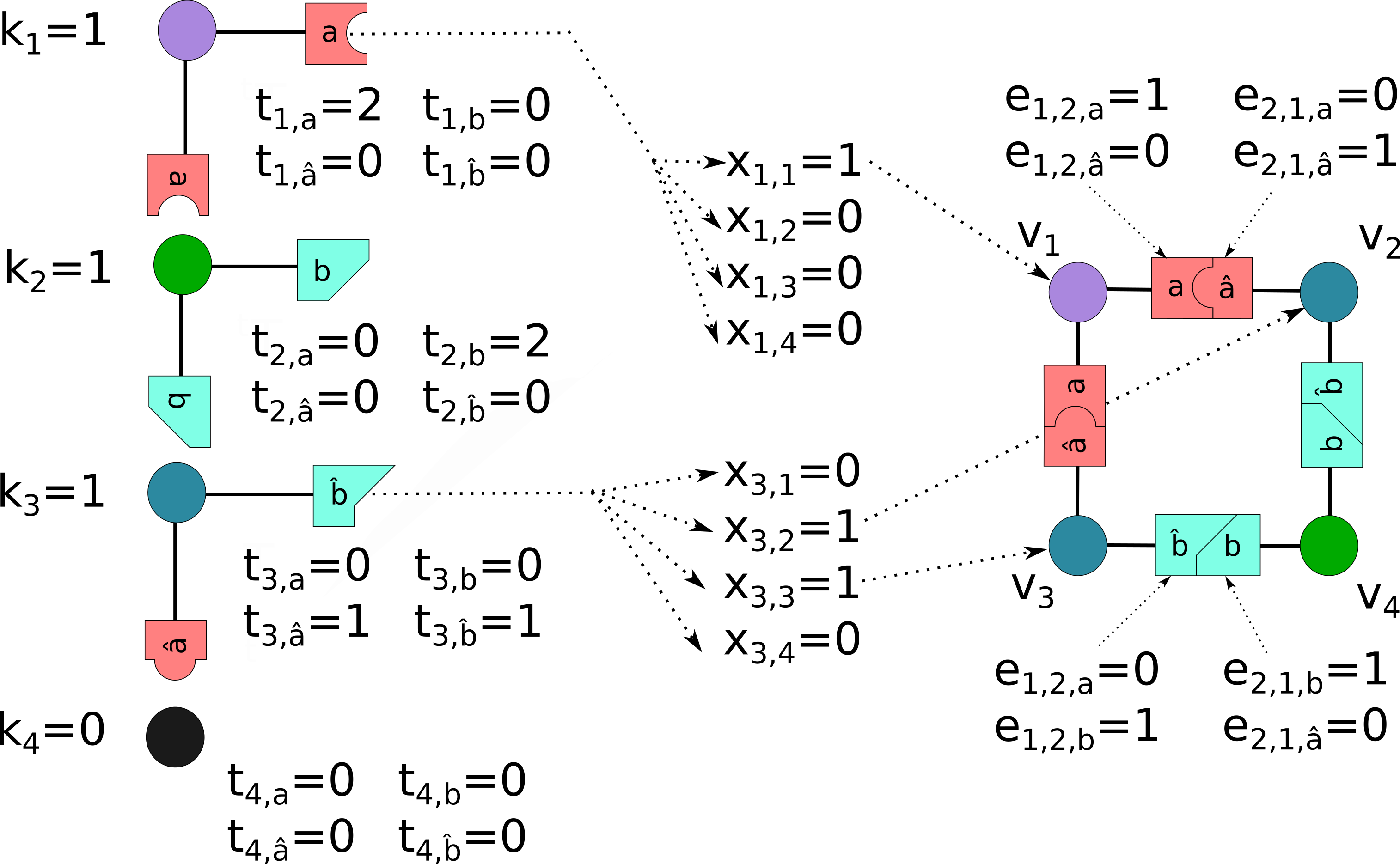}
        \caption{Example realization of $C_4$ with values of decision variables as given in Definition \ref{def:decvar}.}
        \label{fig:c_4realization}
    \end{figure}
\end{example}

The following definition describes several linear constraints encoding the PRP in terms of the decision variables defined above.

\begin{definition} \label{def:PCC}
    The following will be referred to as the \textit{Pot Construction Constraints} on $\beta$ bond-edge types, denoted $PCC(\beta)$. The constraints describe pots realizing a graph $G$ with at most $\beta$ bond-edge types. Here, $(X) \to Y$ denotes that the constraint $Y$ must be satisfied only if the Boolean expression $X$ is true. 
\begin{enumerate}[label=C\arabic*:]
    \item 
    The net number of cohesive-end types is zero 
    \begin{equation*}
        \sum_{i=1}^{N}w_{i,\sigma}-w_{i,\hat{\sigma}}=0\qquad \forall \sigma \in \Sigma_{\beta}
    \end{equation*}
    \item 
    The total number of cohesive-end types on a vertex (i.e. its number of half-edges) is equal to the degree of the vertex:
    \begin{equation*}
        \sum_{\sigma \in \Sigma_{\beta}\cup \widehat{\Sigma}_{\beta}}w_{i,\sigma}=d_i \qquad \forall i\in \mathbb{Z}_{N}
    \end{equation*}
    \item Each vertex is assigned exactly one tile type:
    \begin{equation*}
        \sum_{n=1}^{N}x_{i,n}=1\qquad \forall i\in \mathbb{Z}_{N}
    \end{equation*}
    \item Tile selections are enforced:
    \begin{equation*}
        (x_{i,n}=1)\to w_{i,\sigma}-t_{n,\sigma} \leq 0\qquad \forall i,n\in \mathbb{Z}_{N}, \sigma \in \Sigma_{\beta}\cup \widehat{\Sigma}_{\beta}
    \end{equation*}
    \item Each tile is assigned to a vertex of degree equal to the number of half-edges on the tile (indicator):
    \begin{equation*}
        (x_{i,n}=1)\to \sum_{\sigma \in \Sigma_{\beta}\cup \widehat{\Sigma}_{\beta}}w_{i,\sigma}-t_{n,\sigma}=0 \qquad \forall i,n\in \mathbb{Z}_{N}
    \end{equation*}
     \item $k_n$ decision variables are non-zero if and only if tile $n$ is used (recall $\Delta(G)$ is the maximum degree of a vertex of $G$):
    \begin{equation*}
        -\Delta(G) k_{n}+\sum_{\sigma \in \Sigma_{\beta}\cup \widehat{\Sigma}_{\beta}}t_{n,\sigma}\leq 0\qquad \forall n\in \mathbb{Z}_{N}
    \end{equation*}
    This constraint effectively enforces the following 2 conditions:
    
    1. If tile $t_n$ is empty, $k_n$ must be 0, because the total number of cohesive ends on tile $n$  is 0.
    
    2. If tile $n$ is non-empty, the equation ensures that we get less than 0 on the left-hand side of the inequality for every vertex since the sum cannot be greater than $\Delta$.  
    \item Edge decision variables count the number of incident half-edges of each cohesive end-type:
    \begin{equation*}
        w_{i,\sigma}-\sum_{v_j\in N_G(v_i)}e_{i,j,\sigma}=0\qquad \forall i\in \mathbb{Z}_{N}, \sigma \in \Sigma_{\beta}\cup \widehat{\Sigma}_{\beta}
    \end{equation*}
    \item 
    Each edge is labeled with exactly two cohesive-end types:
    \begin{equation*}
        \sum_{\sigma \in \Sigma_{\beta}\cup \widehat{\Sigma}_{\beta}}e_{i,j,\sigma}+e_{j,i,\sigma}=2\qquad \forall (v_i,v_j)\in E(O(G))
    \end{equation*}
    \item 
    Each edge consists of two half-edges labeled with a hatted and unhatted version of the same bond-edge type:
    \begin{equation*}
        e_{i,j,\sigma}-e_{j,i,\hat{\sigma}}=0 \qquad \forall (v_i,v_j)\in E(G), \sigma \in \Sigma_{\beta}\cup \widehat{\Sigma}_{\beta}
    \end{equation*}
\end{enumerate}
\end{definition}

\begin{remark}
    The constraints are written to allow integer programming solvers as much flexibility as possible in finding the optimal solution, and as such may be less specific than necessary.
\end{remark} 

The following theorem proves that the $PCC$ constraints are equivalent to the constraint that for a given pot $P$ and a given realization function $f$, $\im f = P$, or in other words, that $P$ realizes $G$ via $f$. 

\begin{theorem}
    A pot $P = \{t_1, \cdots t_N\}$ realizes a graph $G$ using $\beta$ bond-edge types if and only $Var(\lambda)$ satisfies $PCC(\beta)$.\label{thm:PCC}
\end{theorem}
\begin{proof} 
    Assume that a pot $P = \{t_1, \cdots t_n\}$ realizes a graph $G$ with realization function $f$. We prove that $Var(\lambda)$ satisfies all of the pot construction constraints given in Definition \ref{def:PCC}. 
     
     C1 follows from Lemma \ref{lemma:P_realize}. C2 follows from the definition of labelled degree. C3, C4, C5, C7, C8, and C9 directly follow from the definition of a realization function (the fact that it is well defined and surjective on $P$).
     
     For C6, we must consider two cases that depend on the value of $k_n$. If $k_n = 0$, then $t_n$ is an empty tile, and $\Sigma_{\sigma \in \Sigma_\beta \cup \hat{\Sigma}_\beta} t_{n,\sigma} = 0$. If $k_n = 1$, then there exists a vertex $v_i$ such that $f(v_i)=t_n$, which implies $x_{i,n}=1$. From C2 and C5 we obtain the following.
     $$\sum_{\sigma \in \Sigma_{\beta}\cup \widehat{\Sigma}_{\beta}} t_{n,\sigma} = \sum_{\sigma \in \Sigma_{\beta}\cup \widehat{\Sigma}_{\beta}}w_{i,\sigma} = d_i $$
     Since $\Delta(G) \geq d_i$,
     $$-\Delta(G) k_{n}+\sum_{\sigma \in \Sigma_{\beta}\cup \widehat{\Sigma}_{\beta}}t_{n,\sigma}  = -\Delta(G) + d_i \leq 0.$$ In either case, we have shown the inequality holds.

     Conversely, now assume that the decision variable tuples \[C = \{(t_{n,\sigma}, w_{i,\sigma}, x_{i,n}, k_{n}, e_{i,j,\sigma} ) \mid \sigma \in \Sigma_{\beta} \cup \widehat{\Sigma}_{\beta}, i,j,n \in \ZZ_{N}\}\] satisfies $PCC(\beta)$. 
     
     We construct assembly design $\lambda$ so that $Var(\lambda) = C$ using the $e_{i, j, \sigma}$ variables.
     
     Let $\lambda((v_i, \{v_i,v_j\}))=a$ if and only if $e_{i,j,a}=1.$

    First, we must show that $\lambda$ is well defined. Thus, we must show $\lambda((v_i, \{v_i,v_j\}))$ is the corresponding hatted symbol to $\lambda((v_j, \{v_i,v_j\}))$  (see Definition \ref{def:assembly_design}). For example, if $\lambda((v_i, \{v_i,v_j\})) = \sigma_1$, then $\lambda((v_j, \{v_i,v_j\})) = \hat{\sigma_1}$
    Rearranging the sum in order to group the hatted and unhatted terms together in C8 gives the following. 
    \begin{equation*}
        \sum_{\sigma \in \Sigma_{\beta}\cup \widehat{\Sigma}_{\beta}}e_{i,j,\sigma} + e_{j,i,\sigma} = \sum_{\sigma \in \Sigma_{\beta}\cup \widehat{\Sigma}_{\beta}}e_{i,j,\sigma} + e_{j,i,\hat{\sigma}} = 2
    \end{equation*}
   C9 ensures that $e_{i,j,\sigma} = e_{j,i,\hat{\sigma}}$, so exactly one of the terms $e_{i,j,\sigma} + e_{j,i,\hat{\sigma}}$ is equal to 2, which shows that both $e_{i,j,\sigma}$ and $e_{j,i,\hat{\sigma}}$ are 1, so  $\lambda$ is well defined.
   
   Now, we show that $Var(\lambda) = C$. The $e_{i,j,a}$ decision variables in $C$ and $Var(\lambda)$ are the same by construction. C7 ensures that $w_{i, \sigma}$ equals the labelled degree of $v_i$, so the $w_{i,a}$ decision variables are the same in $C$ and $Var(\lambda)$. C4, C5, and C6 ensure that the $t_{n,a}, k_n$, and $x_{i,n}$ decision variables match, as well.
    
   C4 also tells us that for each $i$, there is some unique $n_i$ such that $x_{i, n_i} = 1$. Combining this with C5, we get that $t_{n_i,\sigma} = w_{i,\sigma}$. Since $f(v_i)$ is a tile with cohesive ends $w_{i, \sigma}$, $f(v_i) = t_{n_i}$ as desired.
\end{proof}

\subsection{Scenario 1 OPP Algorithm}
\label{subsec:s1_alg}
In Scenario 1, we can restrict ourselves to $\beta = 1$, since all optimal pots in Scenario 1 use only one bond-edge type \cite{ellis2014minimal}. Thus, we employ the following integer program to solve the OPP. 

\begin{equation}
    \min_C \sum_{n=1}^{|G|}k_n \text{ subject to $PCC(1)$} \label{lp:s1}
\end{equation}
where $C$ ranges over all tuples \[\{(t_{n,b}, w_{i,\sigma}, x_{i,n}, k_{n}, e_{i,j,b}): \lambda \in \Sigma_{1} \cup \widehat{\Sigma}_{1}, i,j,n \in \ZZ_{|G|}\}\]. This gives rise to the following algorithm.
\begin{algorithm}
    \caption{Scenario 1 algorithm}
    \text{Input: Graph $G$}
    \text{Output: Pot $P$ such that $G \in \mathcal{O}(P)$}
    \label{alg:S1_ILP}
    \begin{algorithmic}[1]
        \State $C \gets $ the solution to LP \ref{lp:s1} 
        \State return $P = \{t_1, \cdots t_N\}$, where $\forall i: t_i = \{a^{t_{n,a}}\hat{a}^{t_{n, \hat{a}}}\}$ for $t_{n,a}, t_{n,\hat{a}} \in C$
    \end{algorithmic}
\end{algorithm}

\begin{theorem}
   Algorithm \ref{alg:S1_ILP} solves the OPP in Scenario 1.
\end{theorem}
\begin{proof}
    A solution to the Scenario 1 integer program satisfies $PCC(1)$, so by Theorem \ref{thm:PCC} there exists an assembly design $\lambda$ such that $Var(\lambda) = C$ and the pot with tiles $t_n$, where $t_n$ has $t_{n,\sigma}$ cohesive end types of type $\sigma$, realizes $G$. 
    
    
   Now, suppose there is a pot $P'$ realizing $G$ such that $\#P' < \#P$. Let $\lambda'$ be an assembly design corresponding to $P'$. Then, by Theorem \ref{thm:PCC}, $Var(\lambda')$ satisfies $PCC(1)$. The sum of the $k_n$ decision variables in $Var(\lambda')$, which counts the number of tile types of $P'$, would be less than the sum of the $k_n$ decision variables in $Var(\lambda)$, thus contradicting the minimality of $C$.
\end{proof}

\subsection{Scenario 2 OPP Algorithm}
\label{subsec:s2_alg}
Now we consider the Optimal Pot Problem in Scenario 2. An ideal approach to Scenario 2 would combine a pot generation integer program similar to Algorithm \ref{alg:S1_ILP} with the SRP integer program given in (\ref{lp:srp}) into one integer program. Unfortunately, since the two integer programs employ different decision variables, this would require the ability to multiply decision variables, making the problem non-convex. Additionally, the objective functions for the two integer programs are conflicting, which creates a bilevel optimization problem. These two challenges make solving a single Scenario 2 integer program computationally infeasible. Instead, we will not combine Scenario 2 into a single integer program - we will run the integer programs separately. 

In the algorithm presented here, a pot generation integer program similar to Algorithm \ref{alg:S1_ILP} iteratively generates possible realizations of our graph, and then the SRP integer program given in (\ref{lp:srp}) determines whether each generated realization satisfies the constraints of Scenario 2.  If the SRP integer program determines that a pot realizes a graph of smaller order than the target (and hence is not a valid Scenario 2 solution), then the constraints of the pot generation integer program are modified to ensure that the offending pot is not generated again. An overview of this process is shown in Figure \ref{fig:s2sub}.

\begin{figure}
    \centering
    \includegraphics[width=0.5\linewidth]{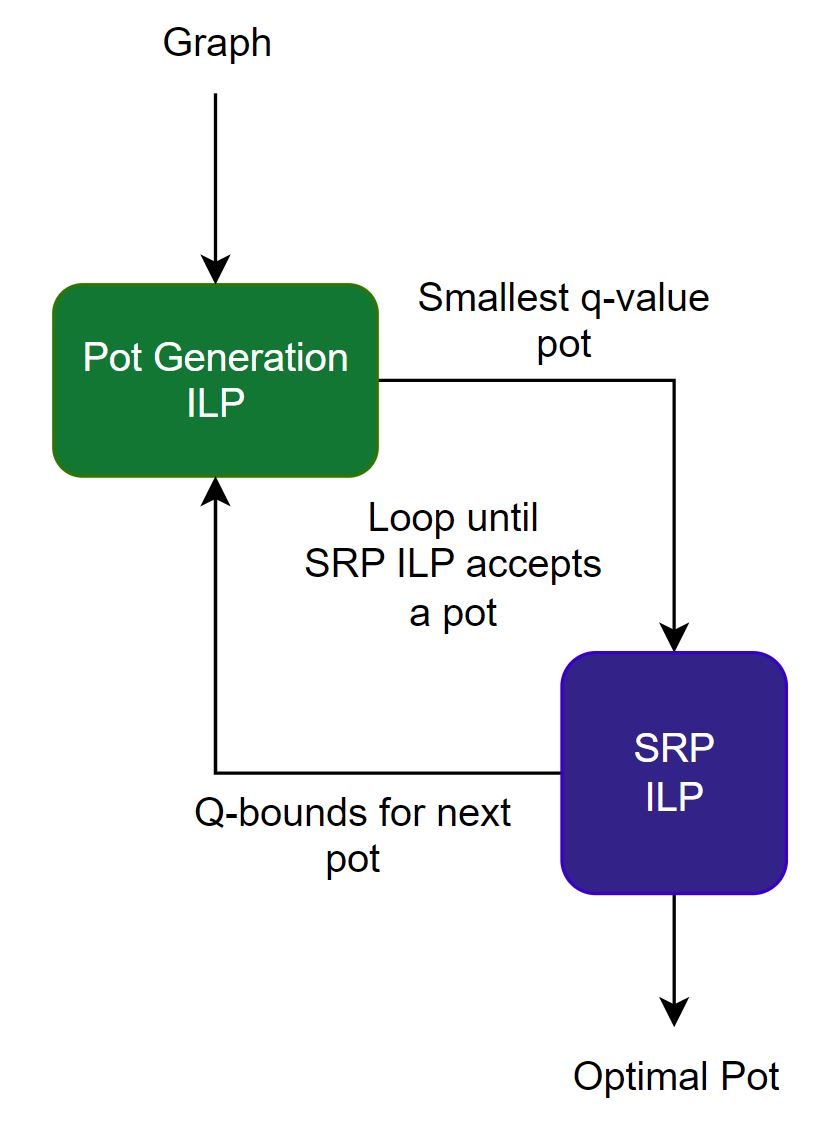}
    \caption{Illustration of Scenario 2 sub-algorithm to determine if there is a pot with $\theta$ bond-edge types and $\beta$ bond-edge types realizing $G$.}
    \label{fig:s2sub}
\end{figure}

In order to efficiently track which pots have already been generated, we seek to create an ordering on the collection of pots. Since pots are defined by their tiles, which are defined by their cohesive-end types, we can fully specify a pot with an integer vector whose entries correspond to the cohesive-end types of its tiles. First, we encode tiles as positional integer vectors. 
\begin{definition}
    The \textit{tile vector} $\vec{v}(t_n)$ of tile $t_n$ in collection $Var(\lambda)$ is the vector whose $i$-th component, where $i$ is odd, is the number of cohesive-ends of the type $\sigma_i$, and whose $j$-th component, where $j$ is even, is the number of cohesive-ends of type $\hat{\sigma}_{j-1}$.
    \begin{equation*}
        \vec{v}(t_n) = \langle t_{n,\sigma_1}, t_{n,\hat{\sigma}_1}, t_{n,\sigma_2}, t_{n,\hat{\sigma}_2}, \cdots t_{n,\sigma_{\beta}}, t_{n,\hat{\sigma}_{\beta}} \rangle.
    \end{equation*}
\end{definition}

To these vectors we can assign an integer in a straightforward manner.
\begin{definition}
     Let the \textit{$q$-value} of $t_n$ be tile vector $\vec{v}(t_n)$ written as an integer in base $\Delta(G)$, or formally \[ \sum_{i=1}^{|v|}v(t_n)[i] * \Delta(G)^{|v| - i} \]
\end{definition}

A candidate for encoding for a pot $P$ is the concatenation of all $\vec{v}(t_n)$ for all $t_n \in P$. But, given that a pot is a set, we should expect all permutations of the tiles to yield the same pot vector. To make the pot vector invariant to permutations, we order the tile vectors lexicographically.

\begin{definition}
    Given any two integer vectors $\vec{v} = \langle v_1, \cdots v_n \rangle $ and $\vec{s} =\langle s_1, \cdots s_n \rangle$, $v <_{lex} s$ if there is an integer $m$ such that 
    \begin{enumerate}
        \item $0 \leq m < n$
        \item $v_i = s_i$ for all $i < m$
        \item $v_m < s_m$
    \end{enumerate} 
\end{definition}

Note that $\vec{v} <_{lex} \vec{s}$ if and only if $q(v) < q(s)$. 

\begin{definition}
     Let $\langle v_1 \mid \cdots \mid v_j \rangle$ denote the concatenation of $v_1, \cdots v_j$. Then the \textit{pot vector} $V(P)$ of $P = \{t_1, \cdots, t_n\}$ is the concatenation of its tile vectors 
    \begin{equation*}
       V(P) = \langle \vec{v}(t_{j_1}) | \cdots | \vec{v}(t_{j_n}) \rangle
    \end{equation*} 
    where $j_1, j_2, \cdots j_n$ is the unique ordering of the tiles in decreasing lexicographic order (i.e. such that $\vec{v}(t_{j_i}) <_{lex} \vec{v}(t_{j_{i-1}})$  for all $2 \leq i \leq n$).
\end{definition}

Now that we have a well-defined encoding of a pot as a vector, we can encode that pot uniquely as an integer. 
\begin{definition}
     The \textit{characteristic value} of $P$, $q(P)$, is $V(P)$ written as an integer in base $\Delta(G)$. That is, \[q(P) = \sum_{i=1}^{\#V}V(P)[i] * \Delta(G)^{\# V - i}.\] \label{def:char_val} 
\end{definition}

\begin{example}
   The pot vector for the pot $P = \{a\hat{c}, \hat{a}\hat{b}, bc\hat{c}, \hat{b}c^2\}$ is as follows.
      \begin{equation*}
        V(P) = \langle 100001 \mid 010100 \mid 001011 \mid 000120 \rangle
     \end{equation*}

    The characteristic value of $P$ in base 10 is the following integer. $$q(P) = 100,001,010,100,001,011,000,120$$ 
\end{example}

We will now use the pot vector to build constraints for a modified pot generation integer program that returns a valid pot with the smallest characteristic value. We can only compare characteristic values for fixed-size vectors, so we design our integer program to search pots with fixed numbers of bond-edge and tile types. To achieve this, we define the Pot Restriction Constraints.

\begin{definition}
     The following will be referred to as the \textit{Pot Restriction Constraints} on $\theta$ tile types and $\beta$ bond-edge types, denoted $PRC(\theta, \beta)$.
\end{definition}
\begin{enumerate}[label=C\arabic*:]
    \item $P$ has exactly $\theta$ tiles:
    \begin{equation*}
        t_{\theta + 1, \sigma} = t_{\theta + 2, \sigma} = \cdots t_{N, \sigma} = 0 \qquad \forall \sigma \in \Sigma_{\beta} \cup  \hat{\Sigma_{\beta}}
    \end{equation*}
    \item Each bond-edge type is used:
    \begin{equation*}
        \sum_{n=1}^{\theta}t_{n,\sigma} \geq 1 \qquad \forall \sigma \in \Sigma_{\beta}
    \end{equation*}
    \item Each tile type is used:
    \begin{equation*}
        \sum_{i=1}^{N}x_{i,n} \geq 1 \qquad \forall n\in\mathbb{Z}_{\theta}
    \end{equation*}
    \item The collection $\{t_{n, \sigma}\}$ is such that the tiles are in correct lexicographic order. That is,
    \begin{equation*}
        q(t_n) - q(t_{n+1}) \geq 1 \qquad \forall n\in\mathbb{Z}_{N-1}
    \end{equation*}
\end{enumerate}

Extending Theorem \ref{thm:PCC} using these constraints is then fairly straightforward. 

\begin{theorem} \label{thm:PRC}
    There is an assembly design $\lambda$ so that vectors $$\vec{t_i} = \langle t_{i, \sigma_1}, t_{i, \hat{\sigma}_1}, \cdots t_{i, \sigma_n}, t_{i, \hat{\sigma}_n} \rangle $$ assemble into a valid pot vector $\langle \vec{t_1} \mid \cdots \mid \vec{t_n} \rangle$ with $\beta$ bond-edge types and $\theta$ tiles if and only if $Var(\lambda)$ satisfies $PCC(\theta)$ and $PRC(\theta, \beta)$.
\end{theorem}

We are now ready to define the Scenario 2 pot generation integer program. The goal of this program is to find a pot with the minimum characteristic value that realizes $G$ with $\beta$ bond-edge types and $\theta$ tile types, such that the characteristic value $q(P)$ is greater than some prespecified bound $Q$:
\begin{align}
    \label{lp:s2}
    &\min_C q(P) \\
    & \text{ subject to $PCC(\beta)$, $PRC(\theta, \beta)$, $q(P) \geq Q$} \nonumber
\end{align}

The algorithm provided below incorporates LP (\ref{lp:s2}) to find a valid Scenario 2 pot given bounds on the number of bond-edge and tile types.
\begin{algorithm}[ht]
\text{Input: A graph $G$, number of bond edge types $\beta$, and number of tile types $\theta$}
\text{Output: A valid Scenario 2 pot with $\beta$ bond-edge types and $\theta$ tile types}
\text{OR an indication that none exists}

\caption{finds valid Scenario 2 pot with $\beta$ bond-edge types and $\theta$ tile types}
\label{alg:s2sub}
\begin{algorithmic}[1] 
    \State $Q \gets 0$
    \While {1 = 1}
        \State $C \gets$ decision variables solving LP \ref{lp:s2} with $q(P) \geq Q$ 
        \If{$C$ is non-empty (a solution to LP \ref{lp:s2} exists)}
            \State $P \gets \{t_1, \cdots, t_N\}$, where $\forall i: t_i = \displaystyle\prod_{\sigma \in \Sigma}\sigma^{t_{n,\sigma}}\hat{\sigma}^{t_{n, \hat{\sigma}}}$
            \State $M \gets$ the smallest order of graph $P$ realizes, found using LP \ref{lp:srp}
            \If{$M = N$}
                \State return $P$
            \Else
                \State $Q \gets q(P) + 1$
            \EndIf
        \Else
            \State There is no pot of size $\theta$ with $\beta$ bond-edge types realizing $G$
        \EndIf
    \EndWhile
\end{algorithmic}
\end{algorithm}

We will now prove the correctness of Algorithm \ref{alg:s2sub}. We will need the following lemma.

\begin{lemma}
    \label{lem:squeeze}
    Let $S = \{P'_1, \cdots P'_n\}$ be the set of all pots $P'$ such that:
    \begin{enumerate}
        \item $P'_i$ realizes $G$ with $\theta$ tile types and $\beta$ bond-edge types.
        \item $P'_i$ is not a valid Scenario 2 pot, i.e. it realizes at least one smaller order graph.
        \item $q(P'_i) < q(P'_{i+1})$ for all $i$ such that $0 < i \leq n$ 
        \item $q(P'_n) < q(\Tilde{P})$ for all $i$ and any valid Scenario 2 pot $\Tilde{P}$.
    \end{enumerate} 
    Then $P_i = P'$ and the algorithm continues to iteration $i + 1$ for all $1 \leq i \leq n$
\end{lemma}

\begin{proof}
    Let $P_i$, $Q_i$, $M_i$ be the values of $P, Q, M$ in the $i$th iteration of Algorithm \ref{alg:s2sub}. 
    
    Now, since $P'_1$ has the smallest characteristic value among all non-valid Scenario 2 pots realizing $G$ and by condition 4, has a smaller characteristic value than any valid Scenario 2 pot realizing $G$, $P_1$ is the solution to Integer Program \ref{lp:s2}, hence $P_1 = P'_1$. Since it is not a valid Scenario 2 pot, the algorithm does not terminate.

    Assume then that $P_j = P'_j$ for $j < i$. Thus, at the $i$th iteration of the algorithm, $Q_{i} = q(P_{i-1}) + 1 = q(P'_{i-1}) + 1$
    
    Consider any pot $P$ satisfying LP \ref{lp:s2} and not equal to $P'_j$ for $j < i$. If $P$ is not a valid S2 pot, then  $q(P'_i) < q(P)$ since otherwise $P = P'_j$ for $j < i$. If it is a valid Scenario 2 pot, condition 4 guarantees that $q(P'_i) < q(P)$. Furthermore, $q(P'_i) > q(P'_{i-1})$ by condition 3. Thus, $P'_i$ has minimal characteristic value among all pots realizing $G$ with characteristic value greater than or equal to $Q_i = q(P'_{i-1}) + 1$, hence $P_i = P'_i$. By induction, the theorem is proved.
\end{proof}

\begin{theorem}
    Algorithm \ref{alg:s2sub} determines if there is a pot $P$ with $\beta$ bond-edge types and $\theta$ tile types such that $G \in \mathcal{O}(P)$ and for all $H \in \mathcal{O}(P)$, $\#V(H) \geq \#V(G)$. If at least one such pot exists, Algorithm \ref{alg:s2sub} returns the pot with the minimum characteristic value $Q^*$.
\end{theorem}
\begin{proof}
    Let $C_i, P_i$, $Q_i$, $M_i$ be the values of $C, P, Q, M$ in the $i$th iteration of Algorithm \ref{alg:s2sub}. 

    First, assume that there is no valid Scenario 2 pot. If there is no pot satisfying LP (\ref{lp:s2}), then, Lemma \ref{thm:PRC} implies that $C_1$ is empty and we will conclude that there is no valid Scenario 2 pot. If there are pots satisfying LP (\ref{lp:s2}), but all of those pots realize at least one graph of smaller order than $G$, then $M_i < N$ in all iterations $i$ of the algorithm. Since we increment $Q$ at each iteration and the maximum characteristic value of all possible pots realizing $G$ is finite, we will eventually create an instance of LP (\ref{lp:s2}) which has no solution, and hence the algorithm will also determine that there is no valid Scenario 2 pot.

    Conversely, assume that there is a valid Scenario 2 pot. Let $S = \{P'_1, \cdots, P'_n\}$ be the set of all pots realizing $G$ with a characteristic value less than $Q^*$ ordered by characteristic value. $S$ satisfies all of the conditions of Lemma \ref{lem:squeeze}, and so there will be an iteration $P_n = P'_{n}$, the element with the maximal characteristic value in $S$. Since $P'_{n} \in S$, $M_i < N$, and hence the algorithm will continue to iteration $i+1$, where $Q_{i+1} > q(P')$ for any $P' \in S$ by the maximality of $P'_{n}$. Therefore, $P_{i+1}$ is not in $S$, which means that it must be a valid Scenario 2 pot realizing $G$. The SRP integer program will verify that $M_{i+1} = N$, so the algorithm returns a pot $P_{i+1}$, which is clearly the valid Scenario 2 pot with the minimal characteristic value.
\end{proof}

We now have an algorithm that determines if there is a valid Scenario 2 pot with $\beta$ bond-edge types and $\theta$ tile types. The final step to develop a full Scenario 2 algorithm is to search the space of possible tuples $(\beta, \theta)$ for an optimal pot. We rely on the following Theorem from \cite{ellis2014minimal} to increase search efficiency.
\begin{theorem}\label{theoremB2T2}
    If $G$ is a graph with $n>2$ vertices, $B_2(G)+1\leq T_2(G)$ 
\end{theorem}

Theorem \ref{theoremB2T2} implies that it is unnecessary to search tuples $(\beta, \theta)$ where $\beta > \theta$. It is sufficient to search the remaining region in order of increasing number of tile types ($\theta)$ to find a pot achieving $(\beta, T_2(G))$ (see Figure \ref{fig:ver_search_space}). It is equivalent to search that same region in order of increasing number of bond-edge types to find a pot achieving $(B_2(G), \theta)$. We combine these procedures in the following algorithm.

\begin{figure}[h]
    \centering
    \includegraphics[width = 0.6\linewidth]{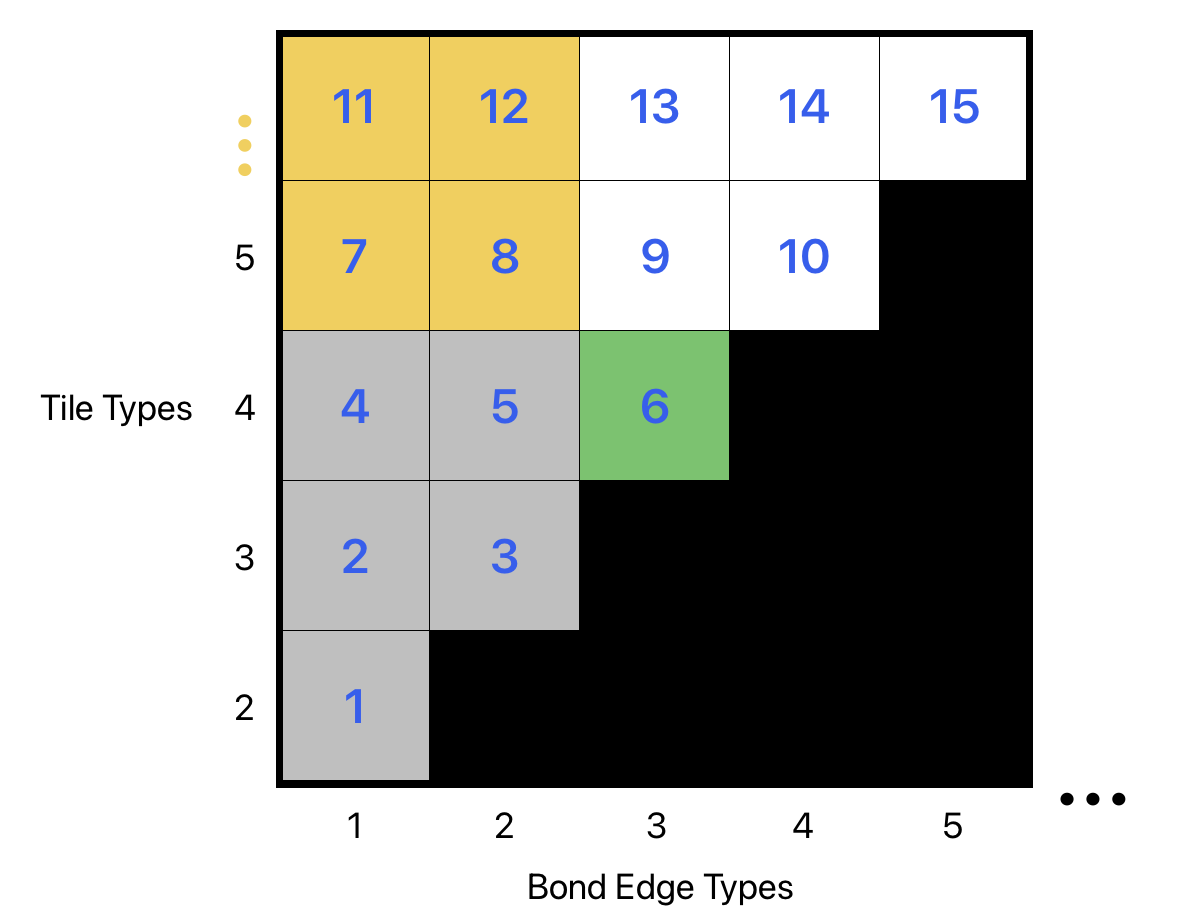}
       \caption{Search space of the Scenario 2 algorithm: the valid search region is searched in the order enumerated in blue: once a solution is found (green cell), we can choose to verify $B_2(G)$ by searching the yellow region}
    \label{fig:ver_search_space}
\end{figure}

\begin{algorithm}[h]
    \caption{Finds $T_2(G)$ for Scenario 2}
    \begin{algorithmic}[1]
        \State find the minimum $\theta$ for which there is an Scenario 2 pot using Algorithm \ref{alg:s2sub} and the search pattern pictured in Figure \ref{fig:ver_search_space}
        \State $T_2(G) = \theta$
        \State among all tuples $(\beta', \theta')$ where $\beta' \leq \beta, \theta' \geq \theta$, find the minimum $\beta'$ for which there is a valid Scenario 2 pot using Algorithm \ref{alg:s2sub}
        \State $B_2(G) = \beta'$
    \end{algorithmic}
\end{algorithm}

\newpage 
\section{Results and Discussion}
\label{sec:results}
We have implemented the algorithms described in Sections \ref{subsec:srp}, \ref{subsec:s1_alg} and \ref{subsec:s2_alg} in Python using the Gurobi MIP solver, which was chosen for its fast runtime and ability to handle many indicator constraints \cite{gurobi, luppold2018evaluating}. The Oriented Optimal Pot Solver ( \href{https://github.com/JacobAshw/OOPS}{OOPS}) for solving the OPP and the Subgraph Realization Problem Solver (SRPS) for solving the SRP are both available along with command-line interfaces and visualization tools to help with ease of use.

\subsection{Implementation of SRP Algorithm}
Even pushing the bounds beyond pots that would realistically be encountered (such as the 27-tile pot or the pot with a 200 half-edge time in Table \ref{table:SRPVerify}), we see that our proposed algorithm is able to compute the minimal realization size of a graph in less than a tenth of a second. This is due to the fact that the SRP integer program is small, with a single constraint for each bond-edge type in the pot. Since integer programs are still relatively easy to solve with constraints numbering in the hundreds (see Table \ref{table:S1Verify} for selected runtimes for the Scenario 1 algorithm), this means the SRP algorithm has a very high ceiling for solvable pots and is not a limiting constraint on the runtime of our larger Scenario 2 algorithm.

\begin{table}
\begin{center}
\caption {SRP Algorithm Verification}
\label{table:SRPVerify}
\scalebox{0.9}{
\begin{tabular}{l|l|l|l|l}
Pot & Pot Size & Minimal Realization Size & Tile Ratio & Runtime (s) \\ \hline
  $a\hat{a},a\hat{b},b\hat{A}$ &  3 & 1 & 1,0,0 & 0.06 \\ \hline
  $\hat{a}\hat{b}\hat{c},cc\hat{b}\hat{b},a\hat{b}\hat{b},b$ & 4 & 8 & 1,1,1,5 & 0.02\\ \hline
  $\hat{a},a^{200}$ & 2 & 201 & 200,1 & 0.02\\ \hline
  $\hat{a},a\hat{b},b\hat{c},\ldots,y\hat{z},z$ & 27 & 27 & $1,1,\ldots,1,1$ & 0.04 \\ \hline
  $abc,\hat{a}d,\hat{b}d,c\hat{d}^3$ & 4 & Invalid Pot & Invalid Pot & 0.01
\end{tabular} }
\end{center}
\end{table}

\subsection{Implementation of Scenario 1 OPP}

We use OOPS to compute optimal pots in Scenarios 1 and 2 for various graphs. This serves three separate purposes:
\begin{enumerate}
    \item Algorithmic verification: We ensure the optimality of known pots.
    \item Runtime analysis: Since the complexity of an integer program varies greatly depending upon the structure of the specific integer program, runtimes for example graphs give a general idea of the scope of graphs for which we can feasibly compute optimal pots.
    \item Discovery of new optimal pots: We search for optimal pots for previously unsolved graphs.
\end{enumerate}

In each scenario, we provide two tables. The first (Tables \ref{table:S1Verify} and \ref{table:S2verify}) compares the algorithm's output on previously solved graphs to ensure that the output aligns with the expected results. The second (Tables \ref{table:S1New} and \ref{table:S2new}) displays computed optimal pots on graphs that were previously unsolved in the specified scenario. The specific graphs are chosen somewhat arbitrarily, but represent a range of possible structures for the purpose of demonstrating the algorithm's potential.

In Scenario 1,  we focus exclusively on optimizing tile types. Table \ref{table:S1Verify} shows our algorithm's computed optimal pots, which demonstrate agreement with the proven $T_1$ values of various other papers. Table \ref{table:S1New} shows the results of the algorithm on a set of various new graphs. These graphs include small examples of social networks, Cartesian products of two different graphs, and disconnected graphs. Each of these sets of graphs was chosen because they represent a possible new area of exploration in finding optimal pots, and they have varying structures that will challenge the algorithm in different ways.

Of the tested graphs, only Les Misérables, $K_{100}$, and Karate Club $\times$ Petersen have runtimes greater than one second. These are all graphs of substantial order, having 77, 100, and 170 vertices respectively. Furthermore, the Les Misérables graph is slower ($\times$1.62) than the $K_{100}$ graph, despite having far fewer edges (242 versus 4950). This suggests that the runtime is generally dependent on the number of vertices rather than the number of edges. This can be explained by the fact that C4 and C5 of the pot construction constraints each enforce $N^2$ constraints, so the number of constraints grows quadratically with the number of vertices.

We have also used this algorithm to find optimal pots in Scenario 1 for disconnected graphs. Since Scenario 1 is not concerned with the construction of smaller graphs, the question of an optimal pot for a disconnected graph is well-posed. At the time of this writing, there are no known previous results for disconnected graphs in the flexible-tile model. The optimal pot is not always simply the union of the optimal pots for each component, as proven by the disconnected graph consisting of $K_4, S_3$, and $P_2$. It is known that the optimal pots of these three graphs in Scenario 1 are $\{a^2\hat{a},a\hat{a}^2\},\{a^3,\hat{a}\},$ and $\{a,\hat{a}\}$, respectively, which under a union create a pot $P$ with $\#\Sigma (P) = 4$ (as shown in Figure \ref{fig:disconnected_independent}). Our algorithm finds a pot $P$ realizing this graph with $P$ with $\#\Sigma (P) = 3$ (as shown in Figure \ref{fig:disconnected_optimal}). This opens the door for exploration of ways in which optimal pots for disconnected graphs relate to the optimal pots of graph's components.

\begin{figure}
\centering
\includegraphics[width=0.3\linewidth]{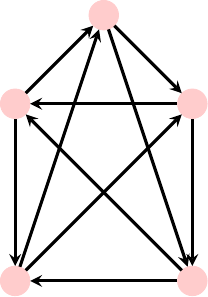}
\quad
\includegraphics[width=0.3\linewidth]{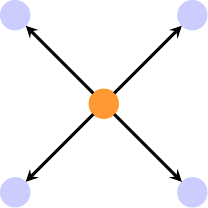}
\quad
\includegraphics[width=0.045\linewidth]{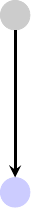}
\caption{Union of separate optimal pots for $T_1(K_5),T_1(S_4),T_1(P_2)$, realizing the graph $K_5\cup S_4\cup P_2$ with the 4-tile pot $\{a^2\hat{a}^2,a^4,\hat{a},a\}$. All edges are labeled $a$ in Scenario 1, so only the orientation is shown.}
\vspace{1em}
\label{fig:disconnected_independent}
\includegraphics[width=0.3\linewidth]{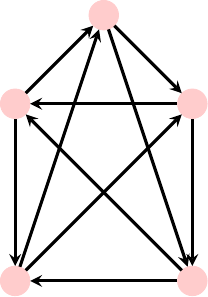}
\quad \includegraphics[width=0.3\linewidth]{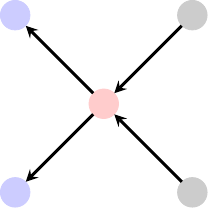}
\quad \includegraphics[width=0.045\linewidth]{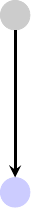}
\caption{Optimal orientation for $T_1(K_5\cup S_4\cup P_2)$, realized with the 3-tile pot $\{a^2\hat{a}^2,\hat{a},a\}$.}
\label{fig:disconnected_optimal}
\end{figure}

\begin{table}
\caption {Scenario 1 OOPS Verification}
\label{table:S1Verify}
\vspace{0.8em}
\scalebox{0.8}{
\begin{tabular}{llllll}
Graph & Vertices & Edges & Known Value &  Computed Pot & Runtime (s) \\ \hline
  Octahedron & 6&12    &        $T_1=1$                              &   $a\hat{a}$          &   0.01                \\ \cite{sorrels} & & & $B_1=1$  \\ \hline
  Lollipop graph $L_{8,10}$ & 18 & 38    &  $T_1=4$         &   $a^3\hat{a}^4, a^7\hat{a}, a\hat{a}, a$          &   0.09                \\ \cite{griffin2023tile} & & & $B_1=1$\\  \hline
  Dodecahedron & 20&30    &       $T_1=2$                              &    $a^3,a\hat{a}^2$         &   0.03                \\ \cite{sorrels} & & & $B_1=1$ \\\hline
  Gear graph $G_{10}$ & 21 & 30      &   $T_1=3$      &     $a\hat{a}^2, a^5, a\hat{a}$        &    0.04               \\ \cite{mattamira2020dna} & & & $B_1=1$  \\ \hline
  $5 \times 5$ Square Lattice  & 25&60    &   $T_1=4$                   &   $\hat{a}^2,a^2\hat{a},\hat{a}^4,a^2\hat{a}^2$          &   0.11                \\ \cite{sorrels} & & & $B_1=1$\\  \hline
  $5 \times 5$ Rook's Graph & 25 & 100 &   $T_1=1$        &      $a^4\hat{a}^4$       &    0.03               \\  \cite{gonzalez2020tile} & & & $B_1=1$ \\ \hline
    Wheel graph $W_{30}$ & 30 & 58     &    $T_1=2$                                    &    $\hat{a}^{15},a^2\hat{a}$         &  0.05                 \\ \cite{lopez2023self} & & & $B_1=1$\\ \hline
  $4 \times 10$ Triangle Lattice & 30&69    &    $T_1=5$            &  $a\hat{a}, a\hat{a}^3, a\hat{a}^2, a^2\hat{a}^3, a^4\hat{a}^2$           &  0.27      \\ \cite{almodovar2021optimal} & & & $B_1=1$ & \\ \hline
  $K_{100}$ & 100 & 4950  &    $T_1=2$                                    &    $\hat{a}^{99},a^{50}\hat{a}^{49}$         &       2.21 \\ \cite{ellis2014minimal} & & & $B_1=1$   
\end{tabular}}
\end{table}

\begin{table}
\begin{center}
\caption {Scenario 1 OOPS New Results} 
\label{table:S1New}
\vspace{0.8em}
\scalebox{0.8}{
\begin{tabular}{llllll}
Graph & Vertices & Edges & Computed & Optimal Pot & Runtime  \\ & & & values & & (seconds)  \\ \hline
$K_5$,$S_4$,$P_2$& 12 & 15 &   $T_1=3$                  &    $a^2\hat{a}^2,\hat{a},a$            &   0.03                             \\ & & & $B_1=1$ \\ \hline
 Tadpole 3 $\times$ $C_4$& 24 & 48 &   $T_1=3$                 &   $a^2\hat{a}^2,a^2\hat{a}^3,a^2\hat{a}$             &    0.12                            \\ & & & $B_1=1$ \\ \hline
  Connected & 25 & 50 & $T_1=3$    &  $a\hat{a}^3,a^3,a^4\hat{a}$              &      0.08                          \\ Caveman (5, 5) & & & $B_1=1$ \\ \hline
 Karate Club & 34 &  79   &    $T_1=11$                 &     $a^6\hat{a}^{10}, a\hat{a}^8, a^3\hat{a}^7, a^3\hat{a}^3, a^2\hat{a}, a\hat{a}^3$           &   0.58                             \\ 
 & & & $B_1=1$ & $a^4\hat{a}, a^2, \hat{a}, a^2\hat{a}^{10}, a^8\hat{a}^9$ & \\ \hline
  All Platonic     & 50 & 90 &     $T_1=5$                  &    $a^2\hat{a}, a\hat{a}^2, a^2\hat{a}^3, a^2\hat{a}^2, a^3\hat{a}^2$            & 0.53 \\             Solids        & & & $B_1=1$ \\ \hline  
 Les Misérables & 77 & 254 &  $T_1=18$                   &     $a, a\hat{a}^9, a^2\hat{a}, a^{11}\hat{a}^{25}, a\hat{a}, a^4\hat{a}^3, a^5\hat{a}^4$           &  3.59                              \\ 
& & & $B_1=1$ & $a^7\hat{a}^8, a^5\hat{a}^6, a^8\hat{a}^8, a^{11}\hat{a}^6, a^2\hat{a}^2, a^2\hat{a}^6$ & \\ 
& & & & $a^4\hat{a}^2, a^{12}\hat{a}^{10}, a^7\hat{a}^{12}, a^8\hat{a}^5, a^9\hat{a}^3$ & \\ \hline
 $3\times 3$-Sudoku & 81 & 810 & $T_1=1$                    &   $a^{10}\hat{a}^{10}$             &  0.34                              \\ & & & $B_1=1$ \\ \hline
 Karate Club  & 340 & 1290 &     $T_1=11$                &        $a^8\hat{a}^4, a^{11}\hat{a}^9, a^7\hat{a}^6, a^5\hat{a}^4, a^3\hat{a}^3, a^4\hat{a}^3$        &    1797.36                            \\ 
 $\times$ Petersen & & & $B_1=1$ & $a^4\hat{a}^4, a^2\hat{a}^3, a^2\hat{a}^2, a^8\hat{a}^7, a^7\hat{a}^{12}$ & \\
      
\end{tabular}}
\end{center}
\end{table}
\subsection{Implementation of Scenario 2 OPP}
In Scenario 2 there are two different versions of optimality: tile optimality and bond-edge optimality. Using the search process described in Section \ref{subsec:s2_alg}, our algorithm first optimizes over tile types, and then verifies the computed value for bond-edge types. In general, we see that the runtime of the algorithm grows with the size of the solution (so values of $T_2(G)$ and $B_2(G)$) rather than the size or order of the graph in question. For instance, the path graph $P_{10}$ has a much longer runtime than the complete graph $K_{100}$ despite having a far smaller size (9 vs. 4950) and order (10 vs 100) (as shown in Table \ref{table:S2verify}). This is because the optimal pot for the path graph consists of six tile types and five bond-edge types, whereas the optimal pot for $K_{100}$ consists of only two tile types and 1 bond-edge type.

In addition to reproducing many known optimal pots in Scenario 2 (Table \ref{table:S2verify}), we use our algorithm to compute many novel optimal pots. One solution of notable interest is the optimal pot for the dodecahedron, as it was the only remaining unsolved platonic solid in Scenario 2 \cite{almodovar2021optimal}. The realization of this graph using the optimal pot can be seen in Fig. \ref{fig:dodecahedron}. Other examples of novel solutions we computed are listed in (Table \ref{table:S2new}). We provide optimal pots but do not provide realizations of each graph here (for the purposes of saving space), although all results can be easily reproduced by downloading and running \href{https://github.com/JacobAshw/OOPS}{OOPS} (found at https://github.com/JacobAshw/OOPS), which displays a visual representation of the graph realization upon computation of the optimal pot. 

It is important to note that the bond verification step of the program only improves upon the solution computed in the tile runtime step when no biminimal pot exists, i.e. the optimal number of tile and bond-edge types cannot be achieved in the same pot. In Scenario 2, an example of a graph that exhibits this property is $Y_{3,3}$ (the stacked prism graph), discovered by Toby Anderson, Olivia Greinke, Iskandar Nazhar, and Luis Santori (pending citation). In practice this property is exceedingly rare (currently $Y_{3,3}$ is the only known example), so the bond verification step nearly never improves upon the pot found in the first step, despite taking much longer.

\begin{figure}
    \centering
    \includegraphics[width=\linewidth]{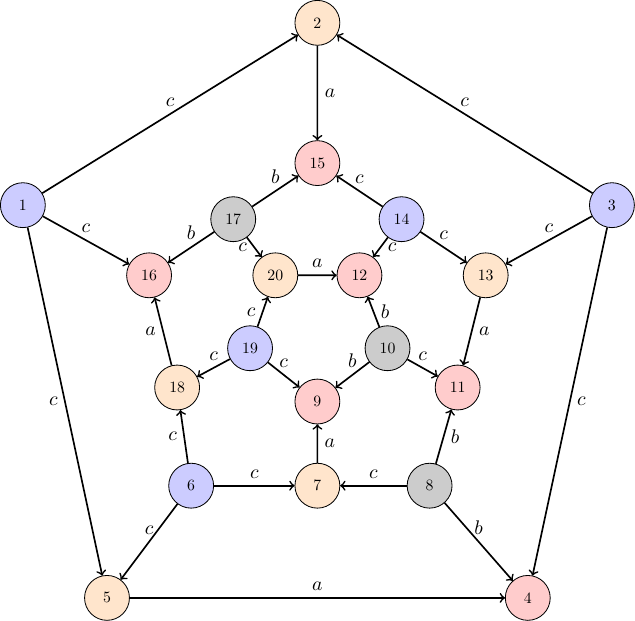}
    \caption{Optimal pot for $T_2$ and $B_2$ realizing the dodecahedron.}
    \label{fig:dodecahedron}
\end{figure}

\begin{table}
{\centering
\caption {Scenario 2 OOPS Verification} 
\vspace{0.8em}
\scalebox{0.8}{
\begin{tabular}{lllp{1.7cm}lll}
Graph & Vertices & Edges & Known & Computed Pot & Tile & Bond \\
 & & & Values & & Runtime (s) & Verification (s) \\ \hline
Tetrahedron      & 4 & 6 &      $T_2=2$    &         $a^2\hat{a},\hat{a}^3$        &      0.02 & 0    \\ \cite{sorrels} & & & $B_2=1$ \\ \hline
Wheel $W_5$      & 5 & 8 &    $T_2=2$   &        $a^2\hat{a},\hat{a}^4$         &  0.02 & 0                 \\ \cite{lopez2023self} & & & $B_2=1$ \\ \hline
Lollipop $L_{3,3}$      & 6 & 6 &  $T_2=5$      & $a\hat{c},\hat{a}\hat{b},bc^2,b$            &        4.29 &  62.66          \\ \cite{griffin2023tile} & & & $B_2=3$ \\ \hline
$2\times 3$ Square      & 6 & 7 &   $T_2=4$  &    $a\hat{b}^2, a\hat{b}, \hat{a}^3, bb$         &    2.18   &    0.25        \\ Lattice \cite{almodovar2021optimal} & & & $B_2=2$ \\ \hline
Path $P_{10}$      & 10 & 9 &     $T_2=6$   &        $a\hat{e}, \hat{a}, be, \hat{b}\hat{d}, cd, \hat{c}d$         & 3233 & 62183                  \\ \cite{ellis2014minimal} & & & $B_2=5$ \\ \hline
Gear $G_5$      & 11 & 15 &     $T_2=3$   &        $a\hat{b},\hat{a}b\hat{b},b^5$     &   0.28  & 15.57              \\ \cite{mattamira2020dna} & & & $B_2=2$ \\ \hline
Icosahedron      & 12 & 30 &        $T_2=3$  &    $a\hat{b}^4, \hat{a}^2b\hat{b}^2, bbb\hat{b}^2$                         &       2.40 & 2.24            \\ \cite{almodovar2021optimal} & & & $B_2=2$  \\ \hline
$K_{100}$      & 100 & 4950 &     $T_2=2$       &  $\hat{a}^{99},a^{50}\hat{a}^{49}$   &  8.96       &    0   \\ \cite{ellis2014minimal} & & & $B_2=1$
\end{tabular}}}
\label{table:S2verify}
\end{table}
\begin{table}
\caption {Scenario 2 OOPS New Results} 
\begin{center}
\vspace{0.8em}
\scalebox{0.8}{
\begin{tabular}{lllp{1.7cm}lll}
Graph & Vertices & Edges & Computed  & Computed  & Tile  & Bond  \\
      &          &       &    values           &  Pot          &  Runtime (s)        &  Verification (s)         \\ \hline
$2\times 4$ Square      & 8 & 10 &   $T_2=4$                    &  $a\hat{c}, \hat{a}\hat{b}, bc\hat{c}, \hat{b}c^2$              &   28.73          &         9744.89          \\ Lattice & & & $B_2=3$ \\ \hline
Turan(8,4)      & 8 & 24 &    $T_2=3$\newline$B_2=2$                    & $a\hat{b}^5, \hat{a}^2b^2\hat{b}^2, b^4\hat{b}^2$               &    1.82         & 5.40                  \\ \hline
$2\times 5$ Square & 10 & 13 & $T_2=5$ & $a\hat{c}, \hat{a}\hat{c}, bc^2, \hat{b}^3, c\hat{c}^2$ & 558.27 & N/A \\ Lattice & & & $B_2\leq 3^*$ \\ \hline
Petersen Graph      & 10 & 15 &   $T_2=3$\newline$B_2=2$     &    $a\hat{b}^2, \hat{a}\hat{b}^2, b^3$            &    0.31         &    0.38               \\ \hline
Mobius Ladder & 10 & 15 & $T_2=3$ & $a\hat{b}^2, \hat{a}^3, b^2\hat{b}$ & 1.25 & 0.48 \\ $M_{10}$ & & & $B_2=2$ \\ \hline
Truncated       & 12 & 18 & $T_2=4$                       &   $a\hat{c}^2, \hat{a}\hat{b}\hat{c}, bc\hat{c}, \hat{b}c^2$             &    161.56         &     N/A              \\ Tetrahedron & & & $ B_2\leq 3^*$ \\ \hline
$3\times 5$ Square      & 15 & 22 &  $T_2=4$  &  $ac\hat{c}, \hat{a}^2\hat{b}^2, \hat{a}b\hat{c}^2, \hat{a}c$   &  97.02           &         N/A          \\ Lattice & & & $B_2\leq 3^*$ \\ \hline
Tesseract & 16 & 32 &$T_2=4$ & $a\hat{c}^3, \hat{a}\hat{c}^3, b^2\hat{b}c, \hat{b}^3c$ & 9509.31 & N/A \\ (4-cube) & & & $B_2\leq 3^*$ \\ \hline
Dodecahedron      & 20 & 30 &  $T_2=4$ & $a\hat{c}^2, \hat{a}\hat{b}\hat{c}, b^2c, c^3$   &  682.83              &     169285                          \\ & & & $B_2=3$

\end{tabular}}
\end{center}
\raggedright{\footnotesize *indicates bond verification was unable to complete due to time constraints}
\label{table:S2new}
\end{table}
\newpage 
\section{Conclusion}
\label{sec:conclusion}
This paper describes three separate computational tools to tackle some of the combinatorial problems associated with DNA self-assembly. The first tool, the SRP integer program, solves the Subgraph Realization Problem for any given pot. The second tool, the Scenario 1 integer program, solves the OPP in Scenario 1 on large graphs. The third tool, the Scenario 2 integer program, solves the OPP in Scenario 2, although it is much slower than the Scenario 1 integer program, and therefore cannot handle graphs with complex optimal pots. These tools provide researchers working on theoretical flexible tile self-assembly with ways to generate and verify solutions that were not previously possible. The SRP integer program allows researchers to verify the optimality of pots in Scenario 2, and the Scenario 1 and 2 integer programs allow for the generation of optimal pots which researchers may be able to generalize to wider families of graphs.

An immediate natural extension of this work is the creation of a Scenario 3 integer program. To do this, two primary challenges must be addressed. First, such a program would require the utilization of a computational tool to solve the non-isomorphic graph realization problem which defines Scenario 3. Once such a tool exists, it can be used similarly to the Scenario 2 integer program to create an additional verification loop in the algorithm. The second challenge is the growing search space. In general, Scenario 3 pots tend to be larger than those in Scenario 2. Since the runtime of our algorithm depends on the complexity of the optimal pot, this will result in significant slowdown. 

There is also opportunity for future work in optimizing the Scenario 2 integer program. Several inefficiencies remain apparent in the algorithm presented here. Many of the different pots generated by the Scenario 2 integer program are equivalent (up to swapping tile labels), which results in the unnecessary checking of a significant number of pots. Additionally, the Scenario 2 integer program is still separated into two scripts. Unifying these into a single script could greatly improve efficiency.

Finally, it is possible to use the algorithms presented here to fully explore optimal solutions across broader families of graphs. Specifically, this could be used to algorithmically generate graphs for which a biminimal pot does not exist in Scenario 2. While it is known that such examples exist, there are no known characterizations of graphs with this property. This is also of great interest to the development of the algorithm itself, as the slowest step of the algorithm is the verification of $B_2(G)$. This step would become unnecessary in the vast majority of cases if it were possible to accurately predict which graphs will exhibit this property.

\section{Acknowledgements}
\label{sec:ack}
We would like to thank faculty and staff at ICERM. Specifically, we want to thank the  the Summer@ICERM TAs, Ryan Firestine, Michaela Fitzgerald, Eric Redmon, and Steph Reyes for their support. We would also like to thank our fellow students who participated in the Summer @ ICERM 2023 REU program.

\section{Funding Sources}
\label{sec:funding}
This material is based upon work supported by the National Science Foundation under Grant No. DMS-1929284 while the authors were in residence at the Institute for Computational and Experimental Research in Mathematics in Providence, RI, during the Summer@ICERM program. 





\bibliographystyle{elsarticle-num} 
\bibliography{jmgm_paper}

\end{document}